\documentclass[reqno,10pt]{amsart}
\usepackage{amssymb}
\usepackage{latexsym}
\usepackage{hyperref}
\usepackage{amsmath}
\usepackage{amscd}
\usepackage{cite}
\usepackage{color}
\usepackage{enumerate}
\usepackage{amsfonts}
\usepackage{graphicx}
\usepackage{mathrsfs}
\usepackage{setspace}

\theoremstyle{plain}
\newtheorem{theorem}{Theorem}
\newtheorem{proposition}[theorem]{Proposition}
\newtheorem{lemma}[theorem]{Lemma}
\newtheorem{corollary}[theorem]{Corollary}

\theoremstyle{definition}

\newtheorem{remark}[theorem]{Remark}

\numberwithin{equation}{section}
\numberwithin{theorem}{section}
\let\Re=\undefined\DeclareMathOperator*{\Re}{Re}
\let\Im=\undefined\DeclareMathOperator*{\Im}{Im}

\def\ge{\geqslant}
\def\le{\leqslant}
\def\geq{\geqslant}
\def\leq{\leqslant}

\def\R{\mathbb{R}}
\def\C{\mathbb{C}}

\begin{document}
\title[3D cubic NLS with a repulsive potential]{Blow-up or grow-up for the focusing 3D cubic NLS with a repulsive inverse-power potential at the mass--energy threshold}
\author[A. Ardila]{Alex H. Ardila}
\address{Alex H. Ardila
\newline \indent Department of Mathematics, Universidad del Valle, Colombia}
\email{ardila@impa.br}
\author[Z. Ma]{Zuyu Ma}
\address{Zuyu Ma\newline \indent Institute of Applied Physics and Computational Mathematics, Beijing 100088, China}
\email{mazuyu23@gscaep.ac.cn}
\author[Y. Wang]{Ying Wang}
\address{Ying Wang \newline \indent Basque Center for Applied Mathematics, Bilbao, Spain}
\email{ywang@bcamath.org}
\begin{abstract}
We consider the focusing cubic nonlinear Schrödinger equation with a repulsive inverse-power potential \(V(x)=a|x|^{-\mu}\), where \(a>0\) and \(1<\mu\leq 2\). At the mass–energy threshold, Miao, Murphy, and Zheng, as well as Ardila, Hamano, and Ikeda, established sharp scattering results in the positive virial region. In this paper, we continue to study the dynamics in the negative virial region and show that, in each time direction, solutions either blow up in finite time or grow up.

\end{abstract}
\maketitle

\section{Introduction}

In this paper, we consider the focusing cubic nonlinear Schr\"odinger
equation with a repulsive inverse-power potential
\begin{equation*}\label{NLSV}
\begin{cases}
i\partial_tu+\Delta u-V(x)u+|u|^2u=0,\\
u|_{t=0}=u_0\in H^1(\mathbb R^3),
\end{cases}
\tag{$\mathrm{NLS}_V$}
\end{equation*}
where $u:I\times\mathbb R^3\to\mathbb C$ and $I$ is its maximal
lifespan. Here
\[
V(x)=a|x|^{-\mu},\qquad a>0,\qquad 1<\mu\leq2.
\]
We define $H=-\Delta+V$ and the homogeneous Sobolev norm adapted to $H$ by
\[
\|u\|_{\dot H_V^1}^2
=\langle Hu,u\rangle
=\int_{\mathbb R^3}\bigl(|\nabla u|^2+V(x)|u|^2\bigr)\,dx.
\]
The Cauchy problem \eqref{NLSV} is locally well-posed in
$H^1(\mathbb R^3)$; see, for example, \cite{GWY,KMVZ}. In addition, its
solutions conserve the mass and energy, defined respectively by
\[
M(u(t)):=\int_{\mathbb R^3}|u(t,x)|^2\,dx
\]
and
\[
E_V(u(t)):=\int_{\mathbb R^3}
\left[\frac12\bigl(|\nabla u(t,x)|^2+V(x)|u(t,x)|^2\bigr)
-\frac14|u(t,x)|^4\right]dx.
\]

To state our main result, we first recall the dynamics of the standard
cubic NLS
\begin{equation*}\label{NLS0}
i\partial_tu+\Delta u+|u|^2u=0,\qquad
(t,x)\in I\times\mathbb R^3,
\tag{$\mathrm{NLS}_0$}
\end{equation*}

The mass-energy level $M(Q)E_0(Q)$  is dictated  by the optimal Gagliardo--Nirenberg inequality \cite{W}
\[
    \|f\|_{L^4}^4
    \leq C_{\mathrm{GN}}
    \|f\|_{L^2}\|\nabla f\|_{L^2}^3,
\]
whose sharp constant is attained by the ground state $Q$.\footnote{Here the ground state $Q$ is known as the unique radial, positive solution to the elliptic equation $-\Delta Q+Q-Q^3=0$.} This variational structure yields the subthreshold dichotomy: if
\[
    M(u_0)E_0(u_0) < M(Q)E_0(Q)
\]
and
\[
    \|u_0\|_{L^2}\|\nabla u_0\|_{L^2}
    <
    \|Q\|_{L^2}\|\nabla Q\|_{L^2},
\]
then the solution is global and scatters, otherwise
finite-time blow-up holds under the  radiality or finite-variance
assumptions; see \cite{DHR, HR}. The sharp Gagliardo--Nirenberg inequality shows that \(M(Q)E_0(Q)\) is the natural mass--energy threshold. Below the threshold, variational estimates yield a uniform separation from the ground-state level, preserving the coercive sign throughout the evolution and leading to the standard subthreshold scattering/blow-up dichotomy. At the threshold, this uniform gap may disappear, allowing solutions to approach the ground-state orbit and exhibit special dynamics that do not occur in the subthreshold regime.
In fact, at the mass--energy threshold
\[
M(u)E_0(u)=M(Q)E_0(Q),
\]

Duyckaerts and Roudenko \cite{DR} constructed two distinguished radial solutions \(Q^+\) and \(Q^-\)  satisfying
$$
M(Q^\pm)=M(Q),\qquad E_0(Q^\pm)=E_0(Q),
$$
and, for some $e_0>0$,
\[
\|Q^\pm(t)-e^{it}Q\|_{H^1}\lesssim e^{-e_0t},
\qquad t\ge0.
\]
Moreover,
$$
\|\nabla Q^-(0)\|_{L^2}<\|\nabla Q\|_{L^2},
\qquad
\|\nabla Q^+(0)\|_{L^2}>\|\nabla Q\|_{L^2},
$$
where \(Q^-\) is global and scatters as \(t\to-\infty\), while \(Q^+\) blows up in finite negative time. 


Together with the standing wave \(e^{it}Q\), these are the exceptional solutions appearing in the threshold classification of Duyckaerts and Roudenko \cite{DR}.

\begin{theorem}[Threshold dynamics for the standard cubic NLS,
{\cite[Theorem~3]{DR}}]
Let $u$ be a maximal-lifespan solution to \eqref{NLS0} satisfying
$M(u)E_0(u)=M(Q)E_0(Q)$.
\begin{enumerate}
\item If
\[
\|u_0\|_{L^2}\|u_0\|_{\dot H^1}
<\|Q\|_{L^2}\|Q\|_{\dot H^1},
\]
then $u$ either scatters in both time directions or agrees with $Q^-$ up
to the symmetries of \eqref{NLS0}.
\item If
\[
\|u_0\|_{L^2}\|u_0\|_{\dot H^1}
=\|Q\|_{L^2}\|Q\|_{\dot H^1},
\]
then $u=e^{it}Q$ up to symmetries.
\item If
\[
\|u_0\|_{L^2}\|u_0\|_{\dot H^1}
>\|Q\|_{L^2}\|Q\|_{\dot H^1}
\]
and $u_0$ is radial or $|x|u_0\in L^2(\R^3)$, then $u$ either blows up in
finite positive and negative time or agrees with $Q^+$ up to symmetries.
\end{enumerate}
\end{theorem}

Thus, under the hypotheses of the preceding theorem and up to the
symmetries of \eqref{NLS0}, the only threshold dynamics
apart from scattering in both time directions and finite-time blow-up
in both time directions are $e^{it}Q$, $Q^+$, and $Q^-$. Campos, Farah,
and Roudenko \cite{CFR} extended this classification to the full
intercritical range.

Without the radiality or finite-variance assumption, a
solution on the supercritical side may instead be global with
unbounded $\dot H^1$ norm.
We say that a solution \(u\) \emph{grows up in positive time} if it is defined on \([0,\infty)\) and
\begin{align}
\limsup_{t\to+\infty}\|u(t)\|_{\dot H^1}=+\infty.
\end{align}
Grow-up in negative time is defined analogously. We also set
$$
K_0(f):=\|f\|_{\dot H^1}^2-\frac34\|f\|_{L^4}^4.
$$
Gustafson and Inui \cite{GI} subsequently removed the radiality and finite-variance assumptions from the supercritical part of the threshold theory. Their result, which is highly relevant to our main theorem, is the following.


\begin{theorem}[Blow-up or grow-up at the threshold,
{\cite[Theorem~1.1]{GI}}]\label{free-blowup-growup}
Let $u$ be a maximal-lifespan solution to \eqref{NLS0} satisfying
\[
 M(u_0)=M(Q),\qquad E_0(u_0)=E_0(Q),\qquad K_0(u_0)<0.
\]
Then $u$ has one of the following five behaviors:
\begin{enumerate}
\item $u$ blows up in both time directions;
\item $u$ blows up in positive time and grows up in negative time;
\item $u$ blows up in negative time and grows up in positive time;
\item $u$ grows up in both time directions;
\item $u$ agrees with $Q^+$ up to the symmetries of \eqref{NLS0}.
\end{enumerate}
\end{theorem}

Nonlinear Schr\"odinger equations with external potentials have been studied
extensively in recent years; see, for example, \cite{{AHI}, {BCMS}, {GWY}, {KMVZ}, {KMVZZ17}, {MMZ}, {Y20}, {Y21}, {YZZ22}, {YZZ} }
and the references therein. In contrast to the standard cubic NLS, the potential breaks space-translation invariance. This loss of symmetry
introduces additional difficulties in the analysis, particularly in
concentration--compactness arguments. Indeed, profiles concentrating far from the spatial origin are only weakly affected by the potential; in the limit, their dynamics are therefore described by the free equation  \eqref{NLS0}. Thus, as in other dispersive problems with broken symmetries,
a key issue is to understand how the dynamics of \eqref{NLSV} 
are related to those of the limiting free problem \eqref{NLS0}.

Without imposing radial symmetry, the sharp constant in the Gagliardo–Nirenberg inequality that is related to \eqref{NLSV} coincides with the free constant $C_\mathrm{GN}$:
$$ \|f\|_{L^4}^4 \leq C_\mathrm{GN}\|f\|_{L^2}\|f\|_{\dot H_V^1}^3. $$
However, when \(a>0\), equality is never attained; see  \cite[Lemma~2.7]{AHI} and
\cite[Theorem~3.1(ii)]{KMVZ}. Indeed, the upper bound follows from
$$ \|f\|_{\dot H^1}\leq \|f\|_{\dot H_V^1}, $$
while the free constant is recovered by translating a free optimizer sufficiently far from the origin, where the potential vanishes at spatial infinity. Consequently, the natural mass--energy threshold remains the free threshold
$$ M(Q)E_0(Q). $$
This same variational structure also governs the dynamics below the threshold. In particular, if
$$ M(u_0)E_V(u_0)<M(Q)E_0(Q) $$
and
$$ \|u_0\|_{L^2}\|u_0\|_{\dot H_V^1} < \|Q\|_{L^2}\|Q\|_{\dot H^1}, $$
then the corresponding solution is global and scatters in \(H^1\); see \cite{{GWY},{KMVZ}}.

For \eqref{NLSV}, Miao, Murphy, and Zheng \cite{MMZ} established threshold scattering when $V(x)=a|x|^{-2}$ with $a>0$, as well as for a class of
short-range repulsive potentials satisfying
\[
V,\ x\cdot\nabla V\in L^{3/2},\qquad
\sup_{x\in\mathbb R^3}\int_{\mathbb R^3}
\frac{|V(y)|}{|x-y|}\,dy<\infty,
\]
and
\[
V\geq0,\qquad x\cdot\nabla V\leq0.
\]
Ardila, Hamano, and Ikeda \cite{AHI} extended the result to
$V(x)=a|x|^{-\mu}$ with $a>0$ and $1<\mu<2$. Define the virial
functional
\begin{equation}\label{virial-functional}
K_V(u):=\|u\|_{\dot H^1}^2
+\frac{\mu}{2}\int_{\mathbb R^3}V(x)|u(x)|^2\,dx
-\frac34\|u\|_{L^4}^4.
\end{equation}
Their threshold scattering result is the following.

\begin{theorem}[Threshold scattering, \cite{AHI,MMZ}]
\label{threshold-scattering}
Let $V(x)=a|x|^{-\mu}$ with $a>0$ and $1<\mu\leq2$. Suppose that
$u_0\in H^1(\mathbb R^3)$ satisfies
\[
M(u_0)E_V(u_0)=M(Q)E_0(Q),\qquad K_V(u_0)>0.
\]
Then the corresponding solution to \eqref{NLSV} is global and belongs to
$L_{t,x}^5(\mathbb R\times\mathbb R^3)$. In particular, there exist
$u_\pm\in H^1(\mathbb R^3)$ such that
\[
\lim_{t\to\pm\infty}
\|u(t)-e^{-itH}u_\pm\|_{H^1}=0.
\]
\end{theorem}
\begin{remark}
    Recently, Yang, Zeng, and Zhang \cite{YZZ} classified the threshold
dynamics for the three-dimensional cubic NLS with an attractive
inverse-square potential.
\end{remark}

Our purpose in this paper is to study the dynamics of \eqref{NLSV} on the negative virial (i.e. $K_V(u_0)<0$) side in the energy space $H_{V}^1$. Under the mass and energy constraints in \eqref{normalized-threshold} below, the condition $K_V(u_0)<0$ is
equivalent to
\[
 \|u_0\|_{\dot H_V^1}>\|Q\|_{\dot H^1};
\]
see  the variational analysis in Section~\ref{preli}.

Our main theorem shows that  every threshold solution on the \(K_V<0\) side either blows up in finite time or grows up in each time direction.
\begin{theorem}\label{main}
Let $V(x)=a|x|^{-\mu}$ with $a>0$ and $1<\mu\leq2$. Let $u$ be a
maximal-lifespan solution to \eqref{NLSV} satisfying
\begin{equation}\label{normalized-threshold}
M(u_0)=M(Q),\qquad E_V(u_0)=E_0(Q),\qquad K_V(u_0)<0.
\end{equation}
Then $u$ has one of the following four behaviors:
\begin{enumerate}
\item $u$ blows up in both time directions;
\item $u$ blows up in positive time and grows up in negative time;
\item $u$ blows up in negative time and grows up in positive time;
\item $u$ grows up in both time directions.
\end{enumerate}
\end{theorem}
Together with the threshold scattering result of
Theorem~\ref{threshold-scattering}, Theorem~\ref{main} completes the
qualitative description of the dynamics at the free ground-state
threshold for the repulsive inverse-power potentials considered here.
Solutions on the $K_V>0$ side scatter, whereas solutions on the
$K_V<0$ side cannot remain global with a uniformly bounded $H^1$
norm: in each time direction, they either blow up in finite time or
grow up. No radiality or finite-variance assumption is imposed.

A notable feature of Theorem~\ref{main} is the absence of a
distinguished solution analogous to the free solution $Q^+$. This is
consistent with the variational structure of the repulsive problem.
As shown in \eqref{sharp-gn}, the sharp constant in the Gagliardo--Nirenberg inequality associated
with $H$ on $H^1(\R^3)$ agrees with the free constant but
is not attained; see \cite[Lemma~2.7]{AHI} and
\cite[Theorem~3.1(ii)]{KMVZ}. Indeed, the free constant can be
recovered by translating a free optimizer sufficiently far from the
origin, where the potential vanishes. Consequently, there is no
variational ground state attaining the free threshold
$M(Q)E_0(Q)$, and Theorem~\ref{main} shows that no $Q^+$-type
alternative occurs at this level.

\begin{remark}
For the inverse-square potential, if one restricts attention to radial
solutions, the optimal constant in the sharp Gagliardo--Nirenberg inequality
changes, giving rise to a larger mass--energy threshold determined by a radial
optimizer. Baker, Campos, Murphy, and Scarpelli \cite{BCMS} recently
classified the radial dynamics at this threshold on both the subcritical
and supercritical sides, including the special solutions analogous to
$Q^\pm$. In contrast, Theorem~\ref{main} concerns the free ground-state
threshold and does not require radiality.
\end{remark}


\begin{remark}
    For finite-variance initial data, Ardila, Hamano, and Ikeda \cite[Theorem~1.3(ii)]{AHI} proved finite-time blow-up when $1<\mu<2$, thereby ruling out the grow-up scenario. In Appendix~\ref{finite-variance-appendix}, we extend their result to the endpoint $\mu=2$ and provide a unified  proof for all $\mu\in(1,2]$.
\end{remark}

A scaling argument gives the following consequence of
Theorem~\ref{main}.

\begin{corollary}\label{product-threshold}
Let $V(x)=a|x|^{-\mu}$ with $a>0$ and $1<\mu\leq2$. Let $u$ be a
maximal-lifespan solution to \eqref{NLSV} satisfying
\[
M(u_0)E_V(u_0)=M(Q)E_0(Q),\qquad K_V(u_0)<0.
\]
Then $u$ has one of the four behaviors in Theorem~\ref{main}.
\end{corollary}

\begin{proof}
Set
\[
\lambda=\frac{M(u_0)}{M(Q)},\qquad
v_0(x)=\lambda u_0(\lambda x),
\]
and define
\[
v(t,x)=\lambda u(\lambda^2t,\lambda x).
\]
Writing
\[
V_\lambda(x)=\lambda^2V(\lambda x)
=a\lambda^{2-\mu}|x|^{-\mu},
\]
we see that $v$ solves $i\partial_t v+\Delta v-V_\lambda(x)v+|v|^2v=0$  and
\[
M(v_0)=M(Q),\qquad E_{V_\lambda}(v_0)=E_0(Q),\qquad
K_{V_\lambda}(v_0)=\lambda K_V(u_0)<0.
\]
Theorem~\ref{main} applies to $v$. Since scaling preserves the four
alternatives in Theorem~\ref{main}, the conclusion follows for $u$.
\end{proof}

\subsection{Outline of the proof of Theorem~\ref{main}}

By the continuation criterion in
Proposition~\ref{local-theory}, conservation of mass, and the
time-reversal symmetry of \eqref{NLSV}, the proof of
Theorem~\ref{main} reduces to the following forward-time statement.

\begin{proposition}\label{prop1.7}
There is no forward-global solution to \eqref{NLSV} satisfying
\eqref{normalized-threshold} and
\begin{equation}\label{bounded-solution}
 \sup_{t\geq0}\|u(t)\|_{H^1}<\infty.
\end{equation}
\end{proposition}

Indeed, if a solution does not blow up in a given time direction, the
continuation criterion implies that it is global in that direction.
If it also does not grow up, then conservation of mass shows that its
$H^1$ norm remains uniformly bounded. After applying time reversal
when necessary, this would contradict Proposition~\ref{prop1.7}.

\subsubsection{Main difficulties}

To facilitate the exposition, we first recall the strategy of Gustafson and Inui \cite{GI} for \eqref{NLS0}. A key ingredient in their argument is the precompactness in \(H^1\), modulo spatial translations, of the forward orbit of an \(H^1\)-bounded threshold solution. This compactness allows one to carry out the modulation and localized virial analysis leading to the classification. For \eqref{NLSV}, however, establishing such precompactness for the entire forward orbit appears substantially more difficult.

The main obstruction arises in the nonlinear profile decomposition. If the spatial parameter of a profile tends to infinity, the potential becomes asymptotically negligible and the limiting dynamics are governed by \eqref{NLS0}. In the threshold scattering argument of \cite{MMZ}, all such limiting profiles lie on the scattering side of the threshold, so they can be embedded globally into solutions of \eqref{NLSV}. In the present setting, by contrast, the dominant profile may lie on the supercritical side and therefore need not scatter, or even be global. A global nonlinear embedding is thus unavailable for this profile. Handling this exceptional profile is the key step in the sequential compactness argument required below.

Our argument avoids this difficulty by requiring only a local-in-time nonlinear embedding for the exceptional profile. The variational estimates together with the decoupling of mass and energy show that at most one profile can fail to scatter. All remaining profiles can therefore be treated by the global nonlinear embedding argument of \cite{MMZ}, while the exceptional profile is approximated only on the fixed compact time interval needed in the sequential compactness argument. This local construction is sufficient for the compactness mechanism used below.

\subsubsection{Strategy of the proof}

To overcome this difficulty, we follow a different strategy from
Gustafson and Inui \cite{GI}. Instead of proving precompactness modulo
spatial translations for the entire forward orbit, we establish
precompactness in $H^1$ only along a suitable sequence of times. This sequential compactness will allow us to extract a forward-global threshold solution that remains in the small modulation regime for all positive times.

The proof of Theorem~\ref{main} proceeds by contradiction. As mentioned before, it is enough to prove Proposition~\ref{prop1.7}. Thus, suppose
that $u$ is a forward-global solution satisfying
\eqref{normalized-threshold} and bounded in $H^1$. We set
\begin{equation}\label{delta-definition}
\delta(f):=\|f\|_{\dot H_V^1}^2-\|Q\|_{\dot H^1}^2,
\qquad \delta(t):=\delta(u(t)).
\end{equation}
Following Du, Wu, and Zhang \cite{DWZ}, a virial argument first yields
a sequence $t_n\to\infty$ such that $\delta(t_n)\to0$.

There is, however, one regime in which we can essentially control the
spatial center, namely, the small modulation regime. When
$0<\delta(t)\ll1$, Proposition~\ref{modulation} shows that the solution
is close in $H^1$ to the orbit of the ground state. In particular, one
can introduce phase, amplitude, and translation parameters whose time
evolutions give control over the spatial center. This motivates our
fundamental strategy: we seek to reduce the problem to the preclusion of
a forward-global threshold solution belonging to the small modulation
regime for all positive times.

\emph{Preclusion of solutions in the small modulation regime.}
To rule out such a solution, we use a modulated virial estimate, that
is, a localized virial estimate centered at the initial position of the
modulated ground state. Together with the modulation estimates, a
bootstrap argument shows that the translation parameter remains within
a bounded distance of its initial position. On the other hand,
Corollary~\ref{sequencec} gives a sequence $t_n\to\infty$ such that
$\delta(t_n)\to0$. The potential-energy estimate from the modulation
analysis then forces the translation parameter to tend to infinity
along this sequence. This contradiction proves Theorem~\ref{rigidity}.

\emph{Reduction to the small modulation regime.}
It remains to show that the solution may be reduced to the preceding
case. If $u$ stays in the small modulation regime after some time
$t_0$, then a time translate of $u$ is ruled out directly by
Theorem~\ref{rigidity}. Otherwise, using the sequence $t_n$ above and a
sequence of times at which $u$ remains a fixed distance from the ground
state orbit, we choose $t_n^+\to\infty$ so that
$\delta(u(t_n^+))=\eta_*$ and $0<\delta(u(t))<\eta_*$ for
$t_n^+<t\leq t_n$.

The main step is to prove that $\{u(t_n^+)\}_n$ is precompact in
$H^1$. Apply the linear profile decomposition to this sequence.
If all associated nonlinear profiles scattered, stability would imply forward scattering of \(u\), which is impossible. On the other hand, the variational estimates and the decoupling of mass and energy show that at most one nonlinear profile can fail to scatter. Hence there is exactly one exceptional profile.
Every remaining profile lies strictly below the
threshold and is treated as in \cite{MMZ}. 
In particular, any of these scattering profiles whose spatial parameters tend to infinity is handled by
Lemma~\ref{nonlinear-embedding-I}. Hence the possibility of several
nonscattering profiles does not arise.


The linear $L^4$ decay rules out $|t_n^1|\to\infty$ for the unique
profile lying on or above the mass--kinetic threshold: otherwise, its
$L^4$ norm tends to zero, and the energy decoupling contradicts the
lower mass--kinetic bound for this profile. If its spatial parameter tends to infinity, the limiting dynamics are governed by \eqref{NLS0}. Unlike the remaining profiles discussed above, this profile may lie on the supercritical side and hence need not scatter or even be global, so the global embedding of  Lemma~\ref{nonlinear-embedding-I} is unavailable. Crucially, however, the compactness argument only requires an approximation on a fixed compact time interval contained in the lifespan of the limiting solution. This is provided by Lemma~\ref{nonlinear-embedding-II}.
The nonlinear profile decomposition and stability then
rule out dichotomy, show that the remainder vanishes, and exclude
spatial escape of the remaining profile. This proves
Proposition~\ref{P:compact}.
Passing to a subsequence, we obtain $u(t_n^+)\to w_0$ in $H^1$. Let
$w$ be the solution with initial data $w_0$. Then the uniform $H^1$ bound for $u$, and the continuation criterion show that $w$ is forward global. Moreover, by a standard stability argument, we can also obtain that $t_n-t_n^+\to\infty$ and  hence
$0<\delta(w(t))\leq\eta_*$ for all $t\geq0$. Thus $w$ is a solution in
the small modulation regime for all positive times, contradicting
Theorem~\ref{rigidity}. This proves Proposition~\ref{prop1.7}; applying
the result to the time-reversed solution completes the proof of
Theorem~\ref{main}.

\subsection{Organization of the paper}

In the rest of the paper, we first collect the preliminary results in
Section~\ref{preli}. In Section~\ref{sec:viride}, we establish the virial
estimates and prove Corollary~\ref{sequencec}. In
Section~\ref{sec:compactness}, we prove the nonlinear embedding lemmas
and Proposition~\ref{P:compact}. The rigidity theorem near the orbit of
$Q$ is proved in Section~\ref{S:rigidity}. Finally,
Section~\ref{S:reduction} completes the proof of Theorem~\ref{main}. The
finite-variance result  is proved in the appendix.

\subsection*{Acknowledgments}
A.~H.~Ardila is partially supported by Universidad del Valle through research project CI-71425. Z. Ma was partially  supported by the NSFC Grant 12271051. Y. Wang is partially supported by the Basque Government through the BERC 2022--2025 program and by the research project PID2024-155550NB-I00 funded by MICIU/AEI/10.13039/501100011033 and FEDER/EU.\,\,
Y. Wang is also supported by a Juan de la Cierva fellowship funded by MICIU/AEI/10.13039/501100011033, under Grant JDC2024-053285-I.

\section{Preliminaries}\label{preli}
For \(s\geq0\) and \(1<r<\infty\), set
\[
 \|f\|_{\dot H_V^{s,r}}=\|H^{s/2}f\|_{L_x^r},
 \qquad
 \|f\|_{H_V^{s,r}}=\|(1+H)^{s/2}f\|_{L_x^r},
\]
and write \(\dot H_V^s=\dot H_V^{s,2}\) and \(H_V^s=H_V^{s,2}\).
We employ the usual $L_t^qL_x^r$ notation for mixed Lebesgue
space-time norms.

We shall use the following equivalence of Sobolev spaces.

\begin{lemma}[Equivalence of Sobolev spaces,
{\cite[Lemma~2.1]{GWY};
\cite[Theorem~1.2]{KMVZZ2}}]\label{sobolev-equivalence}
Let \(0\leq s\leq1\). If \(1<r<\infty\) when \(s=0\), and
\(1<r<3/s\) when \(s>0\), then
\[
 \|(1+H)^{s/2}f\|_{L_x^r}
 \sim \|(1-\Delta)^{s/2}f\|_{L_x^r}.
\]
\end{lemma}

We will also use the estimate
\begin{equation}\label{weighted-hardy}
 \int_{\R^3}V(x+y)|f(x)|^2\,dx
 \lesssim
 \|f\|_{\dot H^1}^{\mu}\|f\|_{L_x^2}^{2-\mu},
 \qquad y\in\R^3,
\end{equation}
which follows from Hardy's inequality and interpolation; see
\cite[Lemma~2.6]{AHI}. In particular,
\(\|f\|_{L_x^2}+\|f\|_{\dot H_V^1}\) is equivalent to the usual
\(H^1\)-norm.

\subsection{Local theory}

We recall the local well-posedness and stability theory for
\eqref{NLSV}.

\begin{proposition}[Well-posedness,
{\cite[Theorem~1.1]{GWY};
\cite[Theorem~2.15 and Remark~2.16]{KMVZ}}]\label{local-theory}
Suppose \(V(x)=a|x|^{-\mu}\), where \(a>0\) and \(1<\mu\leq2\).
\begin{itemize}
\item For any initial data \(u_0\in H^1(\R^3)\), there exists a unique
maximal-lifespan solution \(u\) to \eqref{NLSV}. Any solution that
remains uniformly bounded in \(H^1\) throughout its lifespan is global
in time.
\item Additionally, if
\[
 \|e^{-itH}u_0\|_{L_{t,x}^5((0,\infty)\times\R^3)}
\]
is sufficiently small, then the solution with data \(u_0\) is
forward-global and obeys \(L_{t,x}^5\)-bounds forward in time.
\item More generally, any \(H^1\) solution that remains uniformly
bounded in \(L_{t,x}^5\) throughout its lifespan is global, and global
\(L_{t,x}^5\)-bounds imply scattering.
\item Finally, given any $\psi\in H^1(\R^3)$, we may construct a
solution to \eqref{NLSV} on some interval $(T,\infty)$ that scatters to
$\psi$ in $H^1$.
\end{itemize}
Analogous statements hold backward in time, as well.
\end{proposition}

For the stability results, we introduce the notation
\[
 \dot N^{1/2}(I)
 :=L_t^{10/7}\dot H_x^{1/2,10/7}
  +L_t^{5/3}\dot H_x^{1/2,30/23}
  +L_t^1\dot H_x^{1/2},
\]
\[
 N^{1/2}(I)
 :=L_t^{10/7}H_x^{1/2,10/7}
  +L_t^{5/3}H_x^{1/2,30/23}
  +L_t^1H_x^{1/2},
\]
and
\[
 \dot S^s(I)=L_t^\infty\dot H_V^s\cap L_t^2\dot H_V^{s,6},
 \qquad
 S^s(I)=L_t^\infty H_V^s\cap L_t^2H_V^{s,6},
\]
with all space-time norms over $I\times\R^3$. We obtain the following.

\begin{lemma}[Stability for the inverse-square potential,
{\cite[Theorem~2.17]{KMVZ}}]\label{stability}
Let \(V(x)=a|x|^{-2}\), where \(a>0\), and let
\(I\subset\R\) be an interval containing \(t_0\). Suppose that
\(\widetilde v:I\times\R^3\to\C\) solves
\[
 (i\partial_t-H)\widetilde v
 =-|\widetilde v|^2\widetilde v+e,
 \qquad \widetilde v(t_0)=\widetilde v_0\in H^1,
\]
where $e:I\times\R^3\to\C$.
Let \(v_0\in H^1\), and suppose that
\[
 \|v_0\|_{H^1}+\|\widetilde v_0\|_{H^1}\leq E,
 \qquad
 \|\widetilde v\|_{L_{t,x}^5(I\times\R^3)}\leq L
\]
for some \(E,L>0\). There exists
\(\varepsilon_0=\varepsilon_0(E,L)>0\) such that, if
\(0<\varepsilon<\varepsilon_0\) and
\[
 \|\widetilde v_0-v_0\|_{\dot H^{1/2}}
 +\|e\|_{\dot N^{1/2}(I)}<\varepsilon,
\]
then there exists a solution \(v:I\times\R^3\to\C\) to \eqref{NLSV}
with \(v(t_0)=v_0\) satisfying
\[
 \|v-\widetilde v\|_{\dot S^{1/2}(I)}
 \lesssim_{E,L}\varepsilon,
 \qquad
 \|v\|_{S^1(I)}\lesssim_{E,L}1.
\]
\end{lemma}

\begin{lemma}[Stability for inverse-power potentials,
{\cite[Lemma~2.3]{GWY}}]\label{stability-inverse-power}
Let \(V(x)=a|x|^{-\mu}\), where \(a>0\) and \(1<\mu<2\), and let
\(I\subset\R\) be an interval containing \(t_0\). Suppose that
\(\widetilde v:I\times\R^3\to\C\) solves
\[
 (i\partial_t-H)\widetilde v
 =-|\widetilde v|^2\widetilde v+e,
 \qquad \widetilde v(t_0)=\widetilde v_0\in H^1,
\]
where $e:I\times\R^3\to\C$.
Let \(v_0\in H^1\), and suppose that
\[
 \|v_0\|_{H^1}+\|\widetilde v_0\|_{H^1}\leq E,
 \qquad
 \|\widetilde v\|_{L_{t,x}^5(I\times\R^3)}\leq L
\]
for some \(E,L>0\). There exists
\(\varepsilon_0=\varepsilon_0(E,L)>0\) such that, if
\(0<\varepsilon<\varepsilon_0\) and
\[
 \|\widetilde v_0-v_0\|_{H^{1/2}}
 +\|e\|_{N^{1/2}(I)}<\varepsilon,
\]
then there exists a solution \(v:I\times\R^3\to\C\) to \eqref{NLSV}
with \(v(t_0)=v_0\) satisfying
\[
 \|v-\widetilde v\|_{S^{1/2}(I)}
 \lesssim_{E,L}\varepsilon,
 \qquad
 \|v\|_{S^1(I)}\lesssim_{E,L}1.
\]
\end{lemma}

\subsection{Concentration-compactness}

In this section we import a linear profile decomposition adapted to the
$H^1\to L_{t,x}^5$ Strichartz estimate for $e^{-itH}$. This
decomposition plays a key role in establishing compactness below.

We will use the following consequence of dispersion to rule out profiles
whose time parameters tend to infinity.

\begin{lemma}[{\cite[Lemma~2.8]{GWY};
\cite[Corollary~2.13]{KMVZ}}]
\label{linear-L4-decay}
Let $\phi\in H^1(\R^3)$, let $t_n\to\pm\infty$, and let
$\{x_n\}\subset\R^3$. Then
\[
 \bigl\|e^{it_nH}[\phi(\cdot-x_n)]\bigr\|_{L_x^4}
 \longrightarrow0.
\]
\end{lemma}

\begin{proposition}[Linear profile decomposition,
{\cite[Lemma~2.12]{GWY};
\cite[Proposition~2.4]{MMZ}}]\label{linear-profile}
Let \(\{f_n\}\) be a bounded sequence in \(H^1(\R^3)\). After passing to
a subsequence, there exist
\(J^*\in\{0,1,2,\ldots,\infty\}\), nonzero profiles
\(\{\phi^j\}_{j=1}^{J^*}\subset H^1(\R^3)\), and space-time parameters
\(\{(t_n^j,x_n^j)\}_{j=1}^{J^*}\subset\R\times\R^3\) with the following
properties.

For every finite \(0\leq J\leq J^*\), one has
\begin{equation}\label{linear-profile-expansion}
 f_n=\sum_{j=1}^J\phi_n^j+r_n^J,
 \qquad
 \phi_n^j=e^{it_n^jH}[\phi^j(\cdot-x_n^j)],
 \qquad r_n^J\in H^1(\R^3).
\end{equation}
For every fixed finite \(J\leq J^*\), the following Pythagorean
expansions hold:
\begin{align}
 \|f_n\|_{L_x^2}^2
 &=\sum_{j=1}^J\|\phi_n^j\|_{L_x^2}^2
   +\|r_n^J\|_{L_x^2}^2+o_n(1),
 \label{profile-mass}\\
 \|f_n\|_{\dot H_V^1}^2
 &=\sum_{j=1}^J\|\phi_n^j\|_{\dot H_V^1}^2
   +\|r_n^J\|_{\dot H_V^1}^2+o_n(1),
 \label{profile-kinetic}\\
 \|f_n\|_{L_x^4}^4
 &=\sum_{j=1}^J\|\phi_n^j\|_{L_x^4}^4
   +\|r_n^J\|_{L_x^4}^4+o_n(1).
 \label{profile-quartic}
\end{align}
For every finite \(1\leq J\leq J^*\), the remainder obeys
\begin{equation}\label{linear-remainder-weak}
 \bigl(e^{-it_n^JH}r_n^J\bigr)(\,\cdot+x_n^J)
 \rightharpoonup0
 \qquad\text{weakly in }H^1(\R^3).
\end{equation}
The remainder also satisfies
\begin{equation}\label{linear-remainder-small}
 \lim_{J\to J^*}\limsup_{n\to\infty}
 \|e^{-itH}r_n^J\|_{L_{t,x}^5(\R\times\R^3)}=0.
\end{equation}
The parameters are asymptotically orthogonal in the following sense:
for $j\neq k$,
\begin{equation}\label{profile-orthogonality}
 |t_n^j-t_n^k|+|x_n^j-x_n^k|\longrightarrow\infty.
\end{equation}
Finally, for each $j$, we may assume that either
$t_n^j\equiv0$ or $t_n^j\to\pm\infty$, and either
$x_n^j\equiv0$ or $|x_n^j|\to\infty$.
\end{proposition}

\subsection{Variational analysis}

For the repulsive potential $V(x)=a|x|^{-\mu}$, the sharp
Gagliardo--Nirenberg constant equals the constant $C_0$ for $V=0$, but it is not
attained when $a>0$; see \cite[Lemma~2.7]{AHI} and
\cite[Theorem~3.1(ii)]{KMVZ}. For $V=0$, equality is
attained by the ground state $Q$. Thus
\begin{equation}\label{sharp-gn}
 \|f\|_{L_x^4}^4
 \leq C_0\|f\|_{L_x^2}\|f\|_{\dot H^1}^3
 \leq C_0\|f\|_{L_x^2}\|f\|_{\dot H_V^1}^3.
\end{equation}
The Pohozaev identities give
\begin{equation}\label{pohozaev}
 \|Q\|_{\dot H^1}^2=3\|Q\|_{L_x^2}^2,
 \qquad
 \|Q\|_{L_x^4}^4=4\|Q\|_{L_x^2}^2,
 \qquad
 E_0(Q)=\frac16\|Q\|_{\dot H^1}^2.
\end{equation}
Equivalently, if $M(f)\leq M(Q)$, then
\begin{equation}\label{GN-normalized}
 \|f\|_{L_x^4}^4
 \leq\frac43
 \left(\frac{M(f)}{M(Q)}\right)^{1/2}
 \frac{\|f\|_{\dot H_V^1}}{\|Q\|_{\dot H^1}}
 \|f\|_{\dot H_V^1}^2.
\end{equation}

We shall use the following direct consequence of \eqref{sharp-gn}.  If
$f\in H^1(\R^3)\setminus\{0\}$ and $K_V(f)=0$, then
\begin{equation}\label{zero-virial-threshold}
 M(f)E_V(f)\geq M(Q)E_0(Q).
\end{equation}
Indeed, $K_V(f)=0$ and \eqref{sharp-gn} imply
\[
 \|f\|_{\dot H^1}^2
 \leq\frac34\|f\|_{L_x^4}^4
 \leq\frac34C_0\|f\|_{L_x^2}\|f\|_{\dot H^1}^3,
\]
and hence
\[
 M(f)\|f\|_{\dot H^1}^2
 \geq M(Q)\|Q\|_{\dot H^1}^2.
\]
On the other hand,
\[
 E_V(f)=\frac16\|f\|_{\dot H^1}^2
 +\frac{3-\mu}{6}\int_{\R^3}V|f|^2\,dx.
\]
Thus \eqref{zero-virial-threshold} follows from \eqref{pohozaev}.

We now give the sign properties in the form used below.

\begin{lemma}[Below-threshold sign]\label{below-threshold-sign}
Let $u_0\in H^1(\R^3)$ satisfy
\begin{equation}\label{below-threshold-normalized}
 M(u_0)=M(Q),\qquad E_V(u_0)<E_0(Q),\qquad K_V(u_0)<0.
\end{equation}
Then, throughout the lifespan of the corresponding solution,
\begin{equation}\label{below-threshold-sign-conclusion}
 K_V(u(t))<0,
 \qquad \|u(t)\|_{\dot H_V^1}>\|Q\|_{\dot H^1}.
\end{equation}
Moreover, there exists $c=c(E_0(Q)-E_V(u_0))>0$ such that
\begin{equation}\label{below-threshold-gap}
 \|u(t)\|_{\dot H_V^1}^2\geq
 (1+c)\|Q\|_{\dot H^1}^2
\end{equation}
throughout the lifespan.
\end{lemma}

\begin{proof}
If $K_V(u(t_0))=0$ at some time $t_0$, then
\eqref{zero-virial-threshold} and conservation of mass and energy give
\[
 M(Q)E_0(Q)\leq M(u(t_0))E_V(u(t_0))
 =M(Q)E_V(u_0)<M(Q)E_0(Q),
\]
a contradiction. Thus $K_V(u(t))<0$ by continuity.

Fix a time $t$ in the lifespan and write $f=u(t)$. Since $K_V(f)<0$,
\[
 \|f\|_{\dot H^1}^2
 +\frac\mu2\int_{\R^3}V|f|^2\,dx
 <\frac34\|f\|_{L_x^4}^4.
\]
Using \eqref{sharp-gn}, \eqref{pohozaev}, and $M(f)=M(Q)$, we obtain 
\[
 \|f\|_{\dot H^1}^2
 <\frac34\|f\|_{L_x^4}^4
 \leq \|f\|_{\dot H^1}^2
       \frac{\|f\|_{\dot H^1}}{\|Q\|_{\dot H^1}}.
\]
Hence $\|f\|_{\dot H^1}>\|Q\|_{\dot H^1}$, and therefore
$\|f\|_{\dot H_V^1}>\|Q\|_{\dot H^1}$.

To obtain a uniform gap, set $y(t):=\|u(t)\|_{\dot H_V^1}/\|Q\|_{\dot H^1}$. The previous part shows $y(t)>1$. Since $M(u(t))=M(Q)$, \eqref{GN-normalized} gives $\|u(t)\|_4^4\le\tfrac43y(t)\|u(t)\|_{\dot H_V^1}^2$. Hence
\[
E_V(u_0)=E_V(u(t))
=\tfrac12\|u(t)\|_{\dot H_V^1}^2-\tfrac14\|u(t)\|_4^4
\ge \|Q\|_{\dot H^1}^2(\tfrac12y(t)^2-\tfrac13y(t)^3).
\]
Write $y(t)=1+s(t)$, $s(t)>0$. Since
\[
\tfrac12(1+s)^2-\tfrac13(1+s)^3=\tfrac16-\tfrac12s^2-\tfrac13s^3
\]
and $E_0(Q)=\|Q\|_{\dot H^1}^2/6$, we get
\[
E_V(u_0)\ge E_0(Q)-\|Q\|_{\dot H^1}^2(\tfrac12s(t)^2+\tfrac13s(t)^3).
\]
Thus $d:=E_0(Q)-E_V(u_0)\le \|Q\|_{\dot H^1}^2h(s(t))$, where $h(s)=\tfrac12s^2+\tfrac13s^3$ is strictly increasing on $[0,\infty)$. Let $s_d>0$ solve $h(s_d)=d/\|Q\|_{\dot H^1}^2$. Then $s(t)\ge s_d$ and hence
\[
\|u(t)\|_{\dot H_V^1}^2=\|Q\|_{\dot H^1}^2(1+s(t))^2
\ge \|Q\|_{\dot H^1}^2(1+s_d)^2
=(1+c(d))\|Q\|_{\dot H^1}^2,
\]
with $c(d):=(1+s_d)^2-1>0$. This proves \eqref{below-threshold-gap}.

\end{proof}

\begin{lemma}[Variational dichotomy]\label{variational-dichotomy}
Let $f\in H^1(\R^3)\setminus\{0\}$ satisfy
\[
 M(f)E_V(f)<M(Q)E_0(Q).
\]
Then
\begin{align*}
 \|f\|_{L_x^2}\|f\|_{\dot H_V^1}
  <\|Q\|_{L_x^2}\|Q\|_{\dot H^1}&\quad\Longrightarrow\quad K_V(f)>0,\\
 \|f\|_{L_x^2}\|f\|_{\dot H_V^1}
  \geq\|Q\|_{L_x^2}\|Q\|_{\dot H^1}&\quad\Longrightarrow\quad K_V(f)<0.
\end{align*}
The same statement holds for $V=0$ with $E_0$, $K_0$, and $\dot H^1$.
\end{lemma}
\begin{proof}
By \eqref{zero-virial-threshold}, $K_V(f)$ cannot vanish.  If
$K_V(f)<0$, then \eqref{sharp-gn} gives
\[
 \|f\|_{L_x^2}\|f\|_{\dot H^1}
 >\|Q\|_{L_x^2}\|Q\|_{\dot H^1}.
\]
Thus
$\|f\|_{L_x^2}\|f\|_{\dot H_V^1}
 <\|Q\|_{L_x^2}\|Q\|_{\dot H^1}$ implies $K_V(f)>0$.

If instead
$\|f\|_{L_x^2}\|f\|_{\dot H_V^1}
 \geq\|Q\|_{L_x^2}\|Q\|_{\dot H^1}$, then
\begin{align*}
 M(f)K_V(f)
 &=3M(f)E_V(f)-\frac12M(f)\|f\|_{\dot H_V^1}^2\\
 &\quad-\frac{2-\mu}{2}M(f)\int_{\R^3}V|f|^2\,dx\\
 &<3M(Q)E_0(Q)-\frac12M(Q)\|Q\|_{\dot H^1}^2=0.
\end{align*}
The last identity follows from \eqref{pohozaev}.  This proves the
second implication.  The proof for $V=0$ is identical.
\end{proof}

\begin{theorem}[Sub-threshold scattering, \cite{GWY,KMVZ}]
\label{subthreshold-scattering}
Fix $1<\mu\leq2$ and $a>0$. Let $u$ be the solution to \eqref{NLSV}
with initial data $u_0\in H^1(\R^3)$. If
\begin{equation}\label{subthreshold-scattering-hypothesis}
 M(u_0)E_V(u_0)<M(Q)E_0(Q),\qquad
 \|u_0\|_{L_x^2}\|u_0\|_{\dot H_V^1}
 <\|Q\|_{L_x^2}\|Q\|_{\dot H^1},
\end{equation}
then $u$ exists globally and scatters in $H^1(\R^3)$.
\end{theorem}

The solution in Theorem~\ref{subthreshold-scattering} also satisfies
\[
 \|u\|_{L_{t,x}^5(\R\times\R^3)}
 <C\bigl(E_V(u_0),M(u_0),E_0(Q),M(Q)\bigr).
\]
The corresponding result for \eqref{NLS0} is classical; see
\cite{DHR,HR}.

We next specialize to the threshold and use the distance function
$\delta$ defined in \eqref{delta-definition}.

\begin{lemma}\label{threshold-sign}
Suppose that $M(f)=M(Q)$ and $E_V(f)=E_0(Q)$. Then $\delta(f)\neq0$ and
\begin{equation}\label{delta-virial}
 K_V(f)=-\frac12\delta(f)
 -\frac{2-\mu}{2}\int_{\R^3}V|f|^2\,dx.
\end{equation}
Moreover, $K_V(f)<0$ if and only if $\delta(f)>0$. Consequently, a
solution satisfying \eqref{normalized-threshold} obeys
$\delta(u(t))>0$ and $K_V(u(t))<0$ throughout its lifespan.
\end{lemma}

\begin{proof}
The energy constraint and \eqref{pohozaev} give
\begin{equation}\label{quartic-delta}
 \|f\|_{L_x^4}^4=\|Q\|_{L_x^4}^4+2\delta(f).
\end{equation}
Substituting this identity into the definition of $K_V$ gives
\eqref{delta-virial}. If $\delta(f)=0$, then
\[
 \|f\|_{\dot H_V^1}=\|Q\|_{\dot H^1},
 \qquad \|f\|_{L_x^4}=\|Q\|_{L_x^4}.
\]
Since $V>0$ almost everywhere and $f\neq0$,
$\|f\|_{\dot H^1}<\|Q\|_{\dot H^1}$. This contradicts
\eqref{sharp-gn}. Thus $\delta(f)\neq0$.

If $\delta(f)>0$, then \eqref{delta-virial} gives $K_V(f)<0$. Conversely,
if $K_V(f)<0$, then
\[
 \|f\|_{\dot H^1}^2
 <\frac34\|f\|_{L_x^4}^4
 \leq \|f\|_{\dot H^1}^2
       \frac{\|f\|_{\dot H^1}}{\|Q\|_{\dot H^1}},
\]
so $\|f\|_{\dot H^1}>\|Q\|_{\dot H^1}$ and hence $\delta(f)>0$.
The final assertion follows from conservation and continuity.
\end{proof}

A direct calculation also gives
\begin{equation}\label{energy-virial-general}
 8K_V(f)=24E_V(f)-4\|f\|_{\dot H_V^1}^2
 -4(2-\mu)\int_{\R^3}V|f|^2\,dx.
\end{equation}
In particular, if $E_V(f)\leq E_0(Q)$, then
$8K_V(f)\leq-4\delta(f)$.

Finally, we record the compactness of threshold sequences with
$\delta\to0$.

\begin{lemma}[{\cite[Lemma~4.1 and Remark~4.2]{AHI};
\cite[Lemma~5.2]{MMZ}}]\label{qualitative-modulation}
Suppose that
\[
 M(f_n)=M(Q),\qquad E_V(f_n)=E_0(Q),\qquad \delta(f_n)\to0.
\]
Then there exist $\theta_n\in\R$ and $y_n\in\R^3$ such that
\begin{equation}\label{qualitative-convergence}
 \|f_n-e^{i\theta_n}Q(\cdot-y_n)\|_{H^1}\longrightarrow0,
\end{equation}
\begin{equation}\label{qualitative-potential-center}
 \int_{\R^3}V|f_n|^2\,dx\longrightarrow0,
 \qquad |y_n|\longrightarrow\infty.
\end{equation}
\end{lemma}

\subsection{Modulation near the ground state}

\begin{proposition}[Modulation,
{\cite[Proposition~4.4]{AHI};
\cite[Proposition~5.1]{MMZ}}]\label{modulation}
There exists $\eta_0>0$ with the following property.  Let $u$ satisfy
\eqref{normalized-threshold}.  On the set
$\{t:0<\delta(t)<\eta_0\}$ there are real-valued functions
$\theta,\alpha$ and an $\R^3$-valued function $y$ whose restrictions
to every connected component are $C^1$, such that
\begin{equation}\label{decompz0}
 e^{-i(t+\theta(t))}u(t,x+y(t))=(1+\alpha(t))Q(x)+h(t,x).
\end{equation}
Writing $h=h_1+ih_2$, we have
\begin{equation}\label{modulation-orthogonality}
 \langle h_2,Q\rangle=0,\qquad
 \langle h_1,\partial_jQ\rangle=0\quad (j=1,2,3),\qquad
 \langle h_1,\Delta Q\rangle=0.
\end{equation}
Moreover,
\begin{align}
 &0<\alpha(t)\sim\|h(t)\|_{H^1}
     \sim\|\alpha(t)Q+h(t)\|_{H^1}\sim\delta(t),\label{mo10}\\
 &\int_{\R^3}V(x)|u(t,x)|^2\,dx\lesssim\delta(t)^2,
       \label{potential-modulation}\\
 &|\alpha'(t)|+|y'(t)|\lesssim\delta(t).\label{mo20}
\end{align}
The implicit constants depend only on $a$, $\mu$, and $Q$.
\end{proposition}

The arguments in \cite{AHI,MMZ} also apply on the $K_V<0$ side by using
$|\delta|$; their sign convention for $\delta$ is opposite to ours. In our
convention, the orthogonality conditions and the modulation bounds give
\[
 \delta(t)=2\alpha(t)\|Q\|_{\dot H^1}^2+O(\delta(t)^2),
\]
so $\alpha(t)>0$ when $\eta_0$ is sufficiently small. For $\mu=2$, the
estimate for $\alpha'$ follows by differentiating the orthogonality
conditions as in the proof of \cite[Proposition~4.4]{AHI}; the potential
terms are controlled by \eqref{weighted-hardy} and
\eqref{potential-modulation}.

The explicit form of the inverse-power potential gives the following.

\begin{lemma}\label{power-center-estimate}
Under the assumptions of Proposition~\ref{modulation}, one has
\begin{equation}\label{center-lower-bound}
 (1+|y(t)|)^{-\mu}\lesssim\delta(t)^2.
\end{equation}
\end{lemma}

\begin{proof}
Choose $R>0$ such that $\int_{|x|\leq R}Q^2\,dx>0$.  By
\eqref{mo10}, after decreasing $\eta_0$ we have
\[
 \int_{|x|\leq R}|(1+\alpha)Q+h|^2\,dx
 \geq\frac12\int_{|x|\leq R}Q^2\,dx.
\]
Since $|x+y(t)|\leq R+|y(t)|$ on this ball, it follows that
\[
 \int_{\R^3}V(x)|u(t,x)|^2\,dx
 =\int_{\R^3}V(x+y(t))|(1+\alpha)Q+h|^2\,dx
 \gtrsim(1+|y(t)|)^{-\mu}.
\]
The result now follows from \eqref{potential-modulation}.
\end{proof}
\section{Virial identity and blow-up}\label{sec:viride}

In this section, we prove the following theorem by adapting the argument
of Du, Wu and Zhang \cite[Section~2.2]{DWZ}; see also \cite{GWY}.

\begin{theorem}\label{T:main-1}
Let $u:(T_-,T_+)\times\R^3\to\C$ solve \eqref{NLSV}, with
$M(u)=M(Q)$ and $E_V(u)\leq E_0(Q)$. Suppose that, for some $\kappa>0$,
\begin{equation}\label{equ:assump-1}
 \inf_{0\leq t<T_+}\delta(u(t))\geq\kappa.
\end{equation}
Then $T_+<\infty$, or $T_+=\infty$ and
$\limsup_{t\to\infty}\|u(t)\|_{\dot H^1}=\infty$.
\end{theorem}

We argue by contradiction. Suppose that $u$ is forward global and
\begin{equation}\label{assumcon}
 M(u)=M(Q),\qquad E_V(u)\leq E_0(Q),\qquad
 \inf_{t\geq0}\delta(u(t))\geq\kappa,
 \qquad \sup_{t\geq0}\|u(t)\|_{\dot H^1}\leq C_1.
\end{equation}
Sobolev embedding gives
\begin{equation}\label{assumcon2}
 \sup_{t\geq0}\|u(t)\|_{L_x^6}\leq C_6
\end{equation}
for some $C_6<\infty$.

\subsection{Virial identity}
Given a real weight $w$, define
\[
 I(t;w)=\int_{\R^3}w(x)|u(t,x)|^2\,dx.
\]
For the smooth weights used below, the virial identities are
\begin{align}
 \partial_t I(t;w)
 &=2\Im\int\overline u\,\nabla u\cdot\nabla w\,dx,\label{virial-first}\\
 \partial_{tt} I(t;w)
 &=4\Re\sum_{j,k=1}^3\int w_{jk}\,\overline{u_j}u_k\,dx
   -\int\Delta^2w\,|u|^2\,dx-\int\Delta w\,|u|^4\,dx\notag\\
 &\quad-2\int\nabla w\cdot\nabla V\,|u|^2\,dx.\label{virial}
\end{align}
Here $u_j=\partial_j u$ and
\begin{equation}\label{potential-gradient}
 \nabla V(x)=-a\mu\frac{x}{|x|^{\mu+2}}.
\end{equation}
These formulas follow by approximation from the identities for smooth
solutions; see \cite{AHI,MMZ}. If $\nabla w=O(|x|)$ near the origin, the
potential term is bounded by a multiple of $\int V|u|^2$, which is finite
by \eqref{weighted-hardy}. We will also use weights whose gradient
vanishes near the origin.

\subsection{A localized mass estimate}
For $R>1$, choose a smooth a radial cutoff $\chi_R$ such that
\[
 \chi_R(x)=0\quad(|x|\leq R/2),\qquad
 \chi_R(x)=1\quad(|x|\geq R),
\]
with $0\leq\chi_R\leq1$ and $|\nabla\chi_R|\leq C/R$.
The local conservation of mass gives
\[
 \left|\frac{d}{dt}\int\chi_R|u(t)|^2\,dx\right|
 \leq\frac{C}{R}\|u(t)\|_{L_x^2}\|u(t)\|_{\dot H^1}.
\]
After integration, we obtain the following estimate.

\begin{lemma}\label{lem:copl2}
Under \eqref{assumcon}, for all $t\geq0$ and $R>1$,
\begin{equation}\label{equ:compacl2}
 \int_{|x|\geq R}|u(t,x)|^2\,dx
 \leq\int_{|x|\geq R/2}|u_0(x)|^2\,dx
          +\frac{CC_1M(Q)^{1/2}}{R}\,t.
\end{equation}
\end{lemma}

\subsection{Truncated virial identity}
Let $\phi$ be a smooth, nonnegative radial function such that
\begin{equation}\label{phi-virial}
 \phi(r)=r^2\ (r\leq1),\qquad
 \phi(r)\text{ is constant for }r\geq3,
 \qquad 0\leq\phi'(r)\leq2r,\quad\phi''(r)\leq2.
\end{equation}
Set $w_R(x)=R^2\phi(|x|/R)$ and $I_R(t)=I(t;w_R)$.
In particular, $\nabla^2w_R\leq2\operatorname{Id}$ as quadratic forms,
$|\Delta^2w_R|\lesssim R^{-2}$, and $w_R$ agrees with $|x|^2$ on
$|x|\leq R$.

\begin{lemma}\label{L:virial}
For $R>1$,
\begin{equation}\label{virial-upper}
 I_R''(t)\leq8K_V(u(t))
       +C\int_{|x|>R}\bigl(R^{-2}|u(t,x)|^2+|u(t,x)|^4\bigr)\,dx.
\end{equation}
Under \eqref{assumcon}--\eqref{assumcon2}, there is
$C_2=C_2(\kappa,C_6)>0$ such that
\begin{equation}\label{equ:vttpaest}
 I_R''(t)\leq-3\kappa+C_2\int_{|x|>R}|u(t,x)|^2\,dx.
\end{equation}
\end{lemma}
\begin{proof}
Subtracting $8K_V(u)$ from \eqref{virial} gives
\begin{align*}
 I_R''-8K_V(u)
 &=4\Re\sum_{j,k}\int
       \bigl((w_R)_{jk}-2\delta_{jk}\bigr)\overline{u_j}u_k\,dx
   -\int\Delta^2w_R|u|^2\,dx\\
 &\quad+\int(6-\Delta w_R)|u|^4\,dx
   +2a\mu\int\left(\frac{w_R'(r)}r-2\right)
                     \frac{|u|^2}{r^\mu}\,dx.
\end{align*}
The first and last terms are nonpositive by \eqref{phi-virial}.
The other terms are supported in $|x|>R$ and give
\eqref{virial-upper}.
By \eqref{energy-virial-general}, $8K_V(u(t))\leq-4\kappa$.
Also, H\"older's and Young's inequalities yield
\[
 C\int_{|x|>R}|u|^4\,dx
 \leq CC_6^3\|u\|_{L_x^2(|x|>R)}
 \leq\kappa+C_{\kappa,C_6}\|u\|_{L_x^2(|x|>R)}^2.
\]
Since $R>1$, \eqref{equ:vttpaest} follows.
\end{proof}

\subsection{Proof of Theorem~\ref{T:main-1}}
Combining Lemmas~\ref{lem:copl2} and \ref{L:virial}, and taking $R$
sufficiently large, gives
\begin{equation}\label{virial-cubic-bound}
 I_R''(t)\leq-2\kappa+C_3t/R\qquad(t\geq0),
\end{equation}
where $C_3$ is independent of $R$ and $t$. Thus
\begin{equation}\label{equ:yrest}
 I_R(t)\leq I_R(0)+tI_R'(0)-\kappa t^2+\frac{C_3}{6R}t^3.
\end{equation}
On the other hand,
\begin{align*}
 I_R(0)
 &\lesssim
 \int_{|x|\leq\sqrt R}|x|^2|u_0(x)|^2\,dx
 +R^2\int_{|x|\geq\sqrt R}|u_0(x)|^2\,dx\\
 &\lesssim RM(u_0)
 +R^2\int_{|x|\geq\sqrt R}|u_0(x)|^2\,dx
 =o_R(1)R^2.
\end{align*}
Similarly, by \eqref{virial-first},
\begin{align*}
 |I_R'(0)|
 &\lesssim \sqrt R\|u_0\|_{L_x^2}\|\nabla u_0\|_{L_x^2}
 +R\|u_0\|_{L_x^2(|x|\geq\sqrt R)}
     \|\nabla u_0\|_{L_x^2}\\
 &=o_R(1)R.
\end{align*}
Choose a fixed $\varepsilon>0$ such that $C_3\varepsilon/6\leq\kappa/2$.
Substitute $t=\varepsilon R$ in \eqref{equ:yrest}. We obtain
\[
 0\leq I_R(\varepsilon R)
 \leq-\frac\kappa2\varepsilon^2R^2+o(R^2)<0
\]
for all sufficiently large $R$, a contradiction.
This proves Theorem~\ref{T:main-1}.

\begin{corollary}\label{subthreshold-supercritical}
Suppose that $v_0\in H^1(\R^3)$ satisfies
\begin{equation}\label{subthreshold-supercritical-assumptions}
 M(v_0)E_V(v_0)<M(Q)E_0(Q),
 \qquad K_V(v_0)<0.
\end{equation}
Then, in each time direction, the corresponding solution to
\eqref{NLSV} either blows up in finite time or grows up.
\end{corollary}
\begin{proof}
Let $\lambda=M(v_0)/M(Q)$ and set
\[
 w_0(x)=\lambda v_0(\lambda x),
 \qquad
 V_\lambda(x)=\lambda^2V(\lambda x)
 =a\lambda^{2-\mu}|x|^{-\mu}.
\]
Then
\[
 M(w_0)=M(Q),\qquad E_{V_\lambda}(w_0)<E_0(Q),
 \qquad K_{V_\lambda}(w_0)<0.
\]
Lemma~\ref{below-threshold-sign} gives a constant $c>0$ such that the
corresponding solution satisfies
\[
 \|w(t)\|_{\dot H_{V_\lambda}^1}^2
 \geq(1+c)\|Q\|_{\dot H^1}^2
\]
throughout its lifespan. Theorem~\ref{T:main-1}, applied forward and
backward in time, gives blow-up or grow-up in each direction. Scaling
back proves the result for $v$.
\end{proof}

\begin{corollary}\label{sequencec}
If $u$ is forward global and satisfies \eqref{normalized-threshold} and
\eqref{bounded-solution}, then there exists $t_n\to\infty$ such that
$\delta(u(t_n))\to0$.
\end{corollary}
\begin{proof}
If $\liminf_{t\to\infty}\delta(u(t))>0$, then after a time translation
$u$ satisfies \eqref{equ:assump-1}. Theorem~\ref{T:main-1} contradicts
\eqref{bounded-solution}. Lemma~\ref{threshold-sign} gives
$\delta(u(t))>0$, so the required sequence exists.
\end{proof}

\section{Compactness for nonscattering solutions}\label{sec:compactness}

The next result states that there exist scattering solutions to
\eqref{NLSV} corresponding to initial data below the \eqref{NLS0}
threshold that are translated sufficiently far from the origin.

\begin{lemma}[Nonlinear embedding I]\label{nonlinear-embedding-I}
Suppose that $\{t_n\}\subset\R$ satisfies $t_n\equiv0$ or
$t_n\to\pm\infty$, and let $\{x_n\}\subset\R^3$ satisfy
$|x_n|\to\infty$. Let $\phi\in H^1(\R^3)$ satisfy
\[
 M(\phi)E_0(\phi)<M(Q)E_0(Q),\qquad
 \|\phi\|_{L_x^2}\|\phi\|_{\dot H^1}
 <\|Q\|_{L_x^2}\|Q\|_{\dot H^1}
 \quad\text{if }t_n\equiv0,
\]
and
\[
 \frac12\|\phi\|_{L_x^2}^2\|\phi\|_{\dot H^1}^2
 <M(Q)E_0(Q)
 \quad\text{if }t_n\to\pm\infty.
\]
Define
\[
 \phi_n=e^{it_nH}[\phi(\cdot-x_n)].
\]
Then, for all sufficiently large $n$, there exists a global solution
$v_n$ to \eqref{NLSV} with $v_n(0)=\phi_n$ satisfying
\[
 \|v_n\|_{S^1(\R)}\lesssim1,
\]
where the implicit constant depends on $\phi$. Furthermore,
for every $\varepsilon>0$, there exist
$N_\varepsilon\in\mathbb N$ and
$\psi_\varepsilon\in C_c^\infty(\R\times\R^3)$ such that, whenever
$n\geq N_\varepsilon$,
\[
 \|v_n-\psi_\varepsilon(\,\cdot-t_n,\cdot-x_n)\|_
 {X(\R\times\R^3)}<\varepsilon,
\]
where
\[
 \begin{aligned}
 X\in\{&L_{t,x}^5,\quad L_{t,x}^{10/3},\quad
 L_t^5\dot H_x^{1/2,30/11},\\
 &L_t^{30/7}L_x^{90/31},\quad
 L_t^{30/7}\dot H_x^{31/60,90/31}\}.
 \end{aligned}
\]
\end{lemma}

For $1<\mu<2$, see \cite[Lemma~2.5]{AHI} and
\cite[Lemma~2.13]{GWY}; for $\mu=2$, see
\cite[Theorem~6.1]{KMVZ}.

We will also need the following local version, for which the solution to
\eqref{NLS0} is not assumed to scatter.

\begin{lemma}[Nonlinear embedding II]\label{nonlinear-embedding-II}
Let $\phi\in H^1(\R^3)$, let $|x_n|\to\infty$, and let
$v:I\times\R^3\to\C$ be the maximal-lifespan solution to
\eqref{NLS0} with $v(0)=\phi$. Let $J$ be a compact interval contained
in $I$ and containing zero. Then, for all sufficiently large $n$, there
exists a solution $v_n$ to \eqref{NLSV} on $J$ with
\[
 v_n(0)=\phi(\cdot-x_n)
\]
satisfying
\[
 \|v_n\|_{S^1(J)}\lesssim1,
\]
where the implicit constant may depend on $J$ and $\phi$. We also have
\begin{equation}\label{local-embedding-pointwise}
 \|v_n(\,\cdot,\cdot+x_n)-v\|_{L_t^\infty\dot H_x^{1/2}(J)}
 \longrightarrow0.
\end{equation}
Furthermore, for every $\varepsilon>0$, there exist
$N_\varepsilon\in\mathbb N$ and
$\psi_\varepsilon\in C_c^\infty(\R\times\R^3)$ such that, whenever
$n\geq N_\varepsilon$,
\[
 \|v_n-\psi_\varepsilon(\,\cdot,\cdot-x_n)\|_
 {X(J\times\R^3)}<\varepsilon,
\]
where
\[
 \begin{aligned}
 X\in\{&L_{t,x}^5,\quad L_{t,x}^{10/3},\quad
 L_t^5\dot H_x^{1/2,30/11},\\
 &L_t^{30/7}L_x^{90/31},\quad
 L_t^{30/7}\dot H_x^{31/60,90/31}\}.
 \end{aligned}
\]
In particular, the same conclusions hold on $[0,T]$ if $T>0$ and on
$[T,0]$ if $T<0$, whenever $T\in I$.
\end{lemma}

The proof is a compact-interval adaptation of
\cite[Theorem~6.1]{KMVZ}; compare
\cite[Proposition~2.5]{MMZ} and \cite[Lemma~2.13]{GWY}.

\begin{proof}
We first consider the inverse-square case $\mu=2$.
Fix $0<\theta\ll1$, and let $w_n$ be the
solution to \eqref{NLS0} with
\[
 w_n(0)=P_{|x_n|^\theta}\phi,
\]
where $P_{|x_n|^\theta}$ denotes the Littlewood--Paley projection to
frequencies less than $|x_n|^\theta$. As
$P_{|x_n|^\theta}\phi\to\phi$ strongly in $H^1$, the stability theory
for \eqref{NLS0} shows that $w_n$ is defined on $J$ for all
sufficiently large $n$ and
\[
 \|w_n-v\|_{L_t^\infty H_x^1(J)}
 +\|w_n-v\|_{L_t^2H_x^{1,6}(J)}
 +\|w_n-v\|_{L_{t,x}^5(J\times\R^3)}\longrightarrow0.
\]
Persistence of regularity also gives
\[
 \|w_n\|_{L_t^\infty H_x^2(J)}
 \lesssim_{J,\phi}|x_n|^\theta.
\]

Let $\chi_n$ be a smooth function satisfying
\[
 \chi_n(x)=
 \begin{cases}
 0,&|x+x_n|<\frac14|x_n|,\\
 1,&|x+x_n|>\frac12|x_n|,
 \end{cases}
 \qquad
 |\partial^\alpha\chi_n(x)|\lesssim_\alpha|x_n|^{-|\alpha|},
\]
and define on $J$
\[
 \widetilde v_n(t,x)=[\chi_nw_n](t,x-x_n).
\]
By construction and the dominated convergence theorem,
\[
 \|\widetilde v_n(0)-\phi(\cdot-x_n)\|_{H^1}
 \longrightarrow0,
\]
and
\[
 \limsup_{n\to\infty}
 \left\{
 \|\widetilde v_n\|_{L_t^\infty H_x^1(J)}
 +\|\widetilde v_n\|_{L_{t,x}^5(J\times\R^3)}
 \right\}<\infty.
\]

We claim in addition that
\[
 \bigl\|(i\partial_t-H)\widetilde v_n
   +|\widetilde v_n|^2\widetilde v_n
 \bigr\|_{\dot N^{1/2}(J)}\longrightarrow0.
\]
As in the proof of \cite[Theorem~6.1]{KMVZ}, the error consists of three
types of terms. In the translated coordinates these are
\[
 (\chi_n^3-\chi_n)|w_n|^2w_n,
 \qquad
 w_n\Delta\chi_n+2\nabla\chi_n\cdot\nabla w_n,
 \qquad
 -V(\cdot+x_n)\chi_nw_n.
\]
For the first term, Sobolev embedding and the preceding convergence give
\[
 \|\nabla[(\chi_n^3-\chi_n)|w_n|^2w_n]\|_{L_t^{5/3}L_x^{30/23}(J)}
 \lesssim_{J,\phi}1.
\]
On the other hand,
\begin{align*}
 &\|(\chi_n^3-\chi_n)|w_n|^2w_n\|_{L_t^{5/3}L_x^{30/23}(J)}\\
 &\qquad\lesssim
 \|w_n\|_{L_{t,x}^5(J\times\R^3)}^2
 \|(\chi_n^3-\chi_n)w_n\|_{L_t^5L_x^{30/11}(J)}
 \longrightarrow0.
\end{align*}
Indeed, $w_n\to v$ in $L_t^5L_x^{30/11}(J)$ by interpolation, while
\[
 \|(\chi_n^3-\chi_n)v\|_{L_t^5L_x^{30/11}(J)}\longrightarrow0
\]
by dominated convergence.  Interpolation therefore yields
\[
 \bigl\||\nabla|^{1/2}
 [ (\chi_n^3-\chi_n)|w_n|^2w_n]\bigr\|_{L_t^{5/3}L_x^{30/23}(J)}
 \longrightarrow0.
\]

For the second term, the symbol bounds for $\chi_n$ and the preceding
$H^2$ bound give
\begin{align*}
 \|w_n\Delta\chi_n
   +2\nabla\chi_n\cdot\nabla w_n\|_{L_t^1L_x^2(J)}
 &\lesssim_{J,\phi}(|x_n|^{-2}+|x_n|^{-1}),\\
 \|\nabla(w_n\Delta\chi_n
   +2\nabla\chi_n\cdot\nabla w_n)\|_{L_t^1L_x^2(J)}
 &\lesssim_{J,\phi}
 (|x_n|^{-3}+|x_n|^{-2}+|x_n|^{-1+\theta}).
\end{align*}
Hence
\[
 \bigl\||\nabla|^{1/2}
 [w_n\Delta\chi_n+2\nabla\chi_n\cdot\nabla w_n]
 \bigr\|_{L_t^1L_x^2(J)}\longrightarrow0.
\]
Finally,
\[
 \|V(\,\cdot+x_n)\chi_n\|_{L_x^\infty}
 \lesssim |x_n|^{-\mu},
 \qquad
 \|\nabla[V(\,\cdot+x_n)\chi_n]\|_{L_x^\infty}
 \lesssim |x_n|^{-\mu-1}.
\]
It follows that
\begin{align*}
 \|V(\,\cdot+x_n)\chi_nw_n\|_{L_t^1L_x^2(J)}
 &\lesssim_{J,\phi}|x_n|^{-\mu},\\
 \|\nabla[V(\,\cdot+x_n)\chi_nw_n]\|_{L_t^1L_x^2(J)}
 &\lesssim_{J,\phi}(|x_n|^{-\mu}+|x_n|^{-\mu-1}).
\end{align*}
Interpolating these estimates, we obtain
\[
 \bigl\||\nabla|^{1/2}
 [V(\,\cdot+x_n)\chi_nw_n]\bigr\|_{L_t^1L_x^2(J)}
 \to 0.
\]
Combining the preceding estimates proves the claimed error estimate.

With these estimates in place, Lemma~\ref{stability} yields a
solution $v_n$ to \eqref{NLSV} on $J$ with
$v_n(0)=\phi(\cdot-x_n)$ such that
\[
 \|v_n-\widetilde v_n\|_{\dot S^{1/2}(J)}\to 0,
 \qquad \|v_n\|_{S^1(J)}\lesssim1.
\]

For $1<\mu<2$, the same construction applies with Lemma~\ref{stability-inverse-power}. We first verify the uniform $S^0(J)$ bound needed to apply
Lemma~\ref{stability-inverse-power}. Since $J$ is compact, Sobolev embedding and the uniform $L_t^\infty H_x^1(J)$ bound give
\[
\sup_n\|\widetilde v_n\|_{S^0(J)}\lesssim_J\sup_n\|\widetilde v_n\|_{L_t^\infty H_x^1(J)}<\infty.
\]
We set
\[
e_{n,1}=(\chi_n^3-\chi_n)|w_n|^2w_n,\quad
e_{n,2}=w_n\Delta\chi_n+2\nabla\chi_n\cdot\nabla w_n,\quad
e_{n,3}=-V(\cdot+x_n)\chi_nw_n.
\]
Then
\[
\begin{aligned}
&\|e_{n,1}\|_{L_t^{5/3}L_x^{30/23}(J)}
 +\|e_{n,2}\|_{L_t^1L_x^2(J)}
 +\|e_{n,3}\|_{L_t^1L_x^2(J)}
 \to 0,\\
&\bigl\||\nabla|^{1/2}e_{n,1}\bigr\|_
  {L_t^{5/3}L_x^{30/23}(J)}
 +\bigl\||\nabla|^{1/2}e_{n,2}\bigr\|_
  {L_t^1L_x^2(J)}
 +\bigl\||\nabla|^{1/2}e_{n,3}\bigr\|_
  {L_t^1L_x^2(J)}
 \to 0.
\end{aligned}
\]
Consequently,
\[
 \|e_{n,1}\|_{L_t^{5/3}H_x^{1/2,30/23}(J)}
 +\|e_{n,2}\|_{L_t^1H_x^{1/2}(J)}
 +\|e_{n,3}\|_{L_t^1H_x^{1/2}(J)}
 \to 0.
\]
Hence
\[
 \bigl\|(i\partial_t-H)\widetilde v_n
       +|\widetilde v_n|^2\widetilde v_n
 \bigr\|_{N^{1/2}(J)}
 \to 0.
\]
Moreover, the convergence of the initial data gives $\|\widetilde v_n(0)-\phi(\cdot-x_n)\|_{H^{1/2}}\to0$. Lemma~\ref{stability-inverse-power} then yields
\[
\|v_n-\widetilde v_n\|_{S^{1/2}(J)}\to0,\qquad\|v_n\|_{S^1(J)}\lesssim1.
\]

In any case, by interpolation, Sobolev embedding, and
Lemma~\ref{sobolev-equivalence}, we obtain
\[
 \|v_n-\widetilde v_n\|_{L_{t,x}^5(J\times\R^3)}
 +\bigl\||\nabla|^{1/2}(v_n-\widetilde v_n)
   \bigr\|_{L_t^5L_x^{30/11}(J\times\R^3)}\to 0.
\]
We next verify the remaining approximations in the statement.  Set
$z_n=v_n-\widetilde v_n$.  The sequence $z_n$ is
uniformly bounded in $L_t^\infty H_x^1(J)$ and tends to zero in
$L_t^\infty\dot H_x^{1/2}(J)$.  Since $J$ is compact,
\begin{align*}
 \|z_n\|_{L_{t,x}^{10/3}(J\times\R^3)}
 &\lesssim_J
 \|z_n\|_{L_t^\infty L_x^3}^{4/5}
 \|z_n\|_{L_t^\infty L_x^6}^{1/5}\longrightarrow0,\\
 \|z_n\|_{L_t^{30/7}L_x^{90/31}(J\times\R^3)}
 &\lesssim_J
 \|z_n\|_{L_t^\infty L_x^2}^{1/15}
 \|z_n\|_{L_t^\infty L_x^3}^{14/15}\longrightarrow0.
\end{align*}
The convergence of $w_n$ to $v$, the bounds for $\chi_n$, and the
$S^1(J)$ bound for $v_n$ give
\[
 \limsup_{n\to\infty}
 \|z_n\|_{L_t^{30/7}\dot H_x^{1,90/31}(J\times\R^3)}<\infty.
\]
On the other hand, the convergence of $v_n-\widetilde v_n$ in
$\dot S^{1/2}(J)$, the preceding zero-order estimate, and
Lemma~\ref{sobolev-equivalence} give
\[
 \|z_n\|_{L_t^{30/7}\dot H_x^{1/2,90/31}(J\times\R^3)}
 \longrightarrow0.
\]
Consequently,
\begin{align*}
 &\|z_n\|_{L_t^{30/7}\dot H_x^{31/60,90/31}(J\times\R^3)}\\
 &\qquad\lesssim
 \|z_n\|_{L_t^{30/7}\dot H_x^{1/2,90/31}(J\times\R^3)}^{29/30}
 \|z_n\|_{L_t^{30/7}\dot H_x^{1,90/31}(J\times\R^3)}^{1/30}
 \longrightarrow0.
\end{align*}
The convergence of $w_n$ to $v$, dominated convergence at order zero,
the product rule at order one, and interpolation give the corresponding
convergences with $z_n$ replaced by
$\widetilde v_n-v(\,\cdot,\cdot-x_n)$, as well as
\eqref{local-embedding-pointwise}.  It follows that
\[
 \|v_n-v(\,\cdot,\cdot-x_n)\|_{X(J\times\R^3)}\longrightarrow0
\]
for every space $X$ appearing in the statement.  The conclusion now
follows by choosing a single function in
$C_c^\infty(\R\times\R^3)$ that approximates $v$ simultaneously in
these finitely many spaces.
\end{proof}

\begin{proposition}\label{P:compact}
Suppose that \(u\) is a forward-global solution to \eqref{NLSV}
satisfying \eqref{normalized-threshold} and
\eqref{bounded-solution}.  There exists \(\eta_c>0\) with the following
property.  If \(0<\eta\leq\eta_c\) and \(s_n\to\infty\) satisfy
\[
 \delta(u(s_n))=\eta,
\]
then, after passing to a subsequence, \(u(s_n)\) converges strongly in
\(H^1(\R^3)\).
\end{proposition}
\begin{proof}
We first consider the inverse-square case $\mu=2$.
Set $u_n=u(s_n)$.
By conservation of mass and energy, \eqref{quartic-delta}, and
\eqref{pohozaev},
\begin{equation}\label{fixed-delta-identities}
 \begin{gathered}
 M(u_n)=M(Q),\qquad E_V(u_n)=E_0(Q),\\
 \|u_n\|_{\dot H_V^1}^2=\|Q\|_{\dot H^1}^2+\eta,\qquad
 \|u_n\|_{L_x^4}^4=\frac43\|Q\|_{\dot H^1}^2+2\eta.
 \end{gathered}
\end{equation}
The solution \(u\) does not scatter forward.  Indeed,
Lemma~\ref{threshold-sign} gives \(K_V(u(t))<0\) throughout the
lifespan, whereas a scattering solution has vanishing \(L^4\)-norm and
hence $K_V(u(t))>0$ for all sufficiently large times.

We now apply the linear profile decomposition
(Proposition~\ref{linear-profile}) to the bounded sequence $u_n$ to
obtain
\[
 u_n=\sum_{j=1}^J\phi_n^j+r_n^J,\qquad
 \phi_n^j=e^{it_n^jH}[\phi^j(\cdot-x_n^j)]
\]
along a subsequence, with all of the properties stated in that
proposition. We single out three possibilities, namely: vanishing
($J^*=0$), compactness ($J^*=1$), or dichotomy ($J^*\geq2$).

If $J^*=0$, then we obtain
\[
 \|e^{-itH}u_n\|_{L_{t,x}^5([0,\infty)\times\R^3)}
 \longrightarrow0\qquad\text{as }n\to\infty.
\]
Using this together with stability, we readily deduce
\[
 \|u(s_n+t)\|_{L_{t,x}^5([0,\infty)\times\R^3)}
 =\|u\|_{L_{t,x}^5([s_n,\infty)\times\R^3)}\lesssim1
\]
for all sufficiently large $n$, which contradicts the failure of
forward scattering. Thus vanishing cannot occur.

By a diagonal argument, we may assume that the mass, energy, and
$\dot H_V^1$ norm of every profile and remainder have limits as
$n\to\infty$.  The sharp Gagliardo--Nirenberg inequality may be
written as
\begin{equation}\label{profile-GN}
 \|f\|_{L_x^4}^4
 \leq\frac43
 \left[
 \frac{M(f)\|f\|_{\dot H_V^1}^2}
      {M(Q)\|Q\|_{\dot H^1}^2}
 \right]^{1/2}
 \|f\|_{\dot H_V^1}^2.
\end{equation}
Consequently, the Pythagorean expansions and
\eqref{fixed-delta-identities} imply that, for $\eta_c$ sufficiently
small,
\begin{align}\label{profile-energy-coercivity}
 E_V(\phi_n^j)
 &\geq
 \left(\frac12-\frac13
 \sqrt{1+\frac{\eta}{\|Q\|_{\dot H^1}^2}}\right)
 \|\phi_n^j\|_{\dot H_V^1}^2+o_n(1),\notag\\
 E_V(r_n^J)
 &\geq
 \left(\frac12-\frac13
 \sqrt{1+\frac{\eta}{\|Q\|_{\dot H^1}^2}}\right)
 \|r_n^J\|_{\dot H_V^1}^2+o_n(1)
\end{align}
for every fixed $j$ and $J$.  In particular, all limiting energies
are nonnegative, and the limiting energy of every nonzero profile is
positive.

We next show that one of the profiles lies on or above the
mass--kinetic threshold.  Suppose instead that, for every $j$,
\begin{equation}\label{all-profiles-subthreshold}
 \lim_{n\to\infty}
 M(\phi_n^j)\|\phi_n^j\|_{\dot H_V^1}^2
 <M(Q)\|Q\|_{\dot H^1}^2.
\end{equation}
If $J^*=1$ and $r_n^1\to0$ strongly in $H^1$, then the mass and
kinetic-energy decouplings give
\[
 \lim_{n\to\infty}
 M(\phi_n^1)\|\phi_n^1\|_{\dot H_V^1}^2
 =M(Q)\bigl(\|Q\|_{\dot H^1}^2+\eta\bigr),
\]
contrary to \eqref{all-profiles-subthreshold}.  Thus either
$J^*\geq2$, or $J^*=1$ and $r_n^1$ does not converge strongly to zero
in $H^1$.

The mass and energy decouplings, together with
\eqref{profile-energy-coercivity}, now imply that every profile lies
strictly below the mass--energy threshold:
\begin{equation}\label{profiles-mass-energy-subthreshold}
 \lim_{n\to\infty}M(\phi_n^j)E_V(\phi_n^j)
 <M(Q)E_0(Q),\qquad 1\leq j\leq J^*.
\end{equation}
Indeed, if there is a second profile, then it carries strictly positive
mass and energy.  If $J^*=1$ and the remainder does not converge to
zero in $H^1$, then
\[
 \lim_{n\to\infty}\left[M(r_n^1)+\|r_n^1\|_{\dot H_V^1}^2\right]>0.
\]
If $\lim_nM(r_n^1)>0$, mass decoupling gives
$\lim_nM(\phi_n^1)<M(Q)$. Otherwise,
\eqref{profile-energy-coercivity} gives $\lim_nE_V(r_n^1)>0$, and hence
$\lim_nE_V(\phi_n^1)<E_0(Q)$. Since both limiting quantities for the
first profile are positive and bounded above by $M(Q)$ and $E_0(Q)$,
respectively, their product is strictly smaller than $M(Q)E_0(Q)$.  This proves \eqref{profiles-mass-energy-subthreshold}.

We may therefore construct global nonlinear profiles exactly as in
\cite[Proposition~3.1]{MMZ}.  First suppose that $x_n^j\equiv0$.  If
$t_n^j\equiv0$, let $v^j$ be the solution to \eqref{NLSV} with
$v^j(0)=\phi^j$.  If $t_n^j\to+\infty$, let $v^j$ be the solution
that scatters to $e^{-itH}\phi^j$ as $t\to-\infty$; if
$t_n^j\to-\infty$, then let $v^j$ be the solution
that scatters to $e^{-itH}\phi^j$ as $t\to+\infty$. 
When $t_n^j\equiv0$, the subthreshold scattering theorem applies
directly by \eqref{all-profiles-subthreshold} and
\eqref{profiles-mass-energy-subthreshold}. When $|t_n^j|\to\infty$,
the scattering condition and Lemma~\ref{linear-L4-decay} give
\[
 M(v^j)=M(\phi^j),\qquad
 E_V(v^j)=\frac12\|\phi^j\|_{\dot H_V^1}^2
 =\lim_{n\to\infty}E_V(\phi_n^j).
\]
Moreover,
$\|v^j(-t_n^j)-e^{it_n^jH}\phi^j\|_{H^1}\to0$,
so \eqref{all-profiles-subthreshold} gives
\[
 \|v^j(-t_n^j)\|_{L_x^2}\|v^j(-t_n^j)\|_{\dot H_V^1}
 <\|Q\|_{L_x^2}\|Q\|_{\dot H^1}
\]
for sufficiently large $n$. Applying the subthreshold scattering
theorem at one such time shows that $v^j$ is global and scatters in
both directions.

In each case set
\[
 v_n^j(t,x)=v^j(t-t_n^j,x).
\]
If $|x_n^j|\to\infty$, convergence of the translated operators gives
\[
 E_V(\phi_n^j)\longrightarrow E_0(\phi^j),\qquad
 \|\phi_n^j\|_{\dot H_V^1}^2
 \longrightarrow\|\phi^j\|_{\dot H^1}^2
\]
when $t_n^j\equiv0$.  If $|t_n^j|\to\infty$, then
Lemma~\ref{linear-L4-decay} instead gives
\[
 E_V(\phi_n^j)\longrightarrow
 \frac12\|\phi^j\|_{\dot H^1}^2.
\]
Thus \eqref{all-profiles-subthreshold} and
\eqref{profiles-mass-energy-subthreshold} give the hypotheses of
Lemma~\ref{nonlinear-embedding-I}; let $v_n^j$ be the global solution
furnished by that lemma.  For every fixed $j$,
\begin{equation}\label{scattering-profile-initial-match}
 \|v_n^j(0)-\phi_n^j\|_{H^1}\longrightarrow0.
\end{equation}

Define
\[
 u_n^J(t,x)=\sum_{j=1}^Jv_n^j(t,x)+e^{-itH}r_n^J.
\]
Then
\begin{equation}\label{scattering-initial-approximation}
 \lim_{n\to\infty}\|u_n^J(0)-u_n\|_{H^1}=0
 \qquad\text{for every fixed }J.
\end{equation}
We claim that
\begin{align}
 &\limsup_{J\to J^*}\limsup_{n\to\infty}
 \left\{\|u_n^J(0)\|_{H^1}
 +\|u_n^J\|_{L_{t,x}^5(\R\times\R^3)}\right\}
 \lesssim1,\label{scattering-profile-bounds}\\
 &\limsup_{J\to J^*}\limsup_{n\to\infty}
 \bigl\|(i\partial_t-H)u_n^J+|u_n^J|^2u_n^J
 \bigr\|_{\dot N^{1/2}(\R)}=0.
 \label{scattering-equation-error}
\end{align}

Before turning to the proof of \eqref{scattering-profile-bounds} and
\eqref{scattering-equation-error}, we observe that by the orthogonality
of the parameters and approximation by functions in
$C_c^\infty(\R\times\R^3)$ (see Lemma~\ref{nonlinear-embedding-I} for
the case $|x_n^j|\to\infty$), we may obtain the following: for any
$j\ne k$,
\begin{equation}\label{scattering-profile-orthogonality}
 \|v_n^jv_n^k\|_{L_{t,x}^{5/2}}
 +\|v_n^j|\nabla|^{1/2}v_n^k
 \|_{L_t^{5/2}L_x^{30/17}}
 \longrightarrow0.
\end{equation}
Here and below, the norms in this part of the proof are taken over
$\R\times\R^3$.
For the treatment of the linear remainder, the same compact
approximation argument also gives
\begin{equation}\label{scattering-auxiliary-orthogonality}
 \begin{aligned}
 &\|v_n^jv_n^k\|_{L_{t,x}^{5/3}}
 +\|v_n^jv_n^k\|_{L_t^{15/7}L_x^{45/31}}\\
 &\quad+
 \|(|\nabla|^{31/60}v_n^j)(|\nabla|^{31/60}v_n^k)
   \|_{L_t^{15/7}L_x^{45/31}}
 \longrightarrow0
 \end{aligned}
 \qquad(j\ne k).
\end{equation}

\begin{proof}[Proof of \eqref{scattering-profile-bounds}]
The $H^1$ bound follows from
\eqref{scattering-initial-approximation}.  For the uniform $S^1$
bound, we follow the proof of \cite[(7.14)]{KMVZ}.  By the $H^1$
decoupling,
\[
 \limsup_{n\to\infty}\sum_{j=1}^J
 \|\phi_n^j\|_{H^1}^2\lesssim1
\]
uniformly in $J$.  The definition of the profiles and the convergence
of the translated operators therefore give
\[
 \sum_{j=1}^{J^*}\|\phi^j\|_{H^1}^2<\infty.
\]
If $J^*<\infty$, the square-sum bound below follows immediately.
Suppose therefore that $J^*=\infty$, and choose $J_0$ sufficiently
large that the last sum restricted to $j\geq J_0$ lies below the
small-data threshold.  For these profiles, the small-data estimate
$\|v_n^j\|_{S^1(\R)}\lesssim\|\phi^j\|_{H^1}$ holds uniformly for
large $n$.  Hence
\[
 \sup_{J\geq J_0}\limsup_{n\to\infty}
 \sum_{j=J_0}^J\|v_n^j\|_{S^1(\R)}^2
 \lesssim\sum_{j=J_0}^{J^*}\|\phi^j\|_{H^1}^2.
\]
For each $j<J_0$, the corresponding nonlinear profile has a finite
$S^1(\R)$ norm.  There are only finitely many such profiles, and hence
\[
 \limsup_{J\to J^*}\limsup_{n\to\infty}
 \sum_{j=1}^J\|v_n^j\|_{S^1(\R)}^2\lesssim1.
\]
For every fixed $J$, \eqref{scattering-profile-orthogonality} also
gives
\begin{align*}
 &\left|
 \left\|\sum_{j=1}^Jv_n^j\right\|_{L_{t,x}^5}^5
 -\sum_{j=1}^J\|v_n^j\|_{L_{t,x}^5}^5\right|\\
 &\qquad\lesssim_J
 \sum_{j\ne k}\|v_n^j\|_{L_{t,x}^5}^3
 \|v_n^jv_n^k\|_{L_{t,x}^{5/2}}\longrightarrow0.
\end{align*}
Together with \eqref{scattering-profile-orthogonality}, Strichartz,
and \eqref{linear-remainder-small}, this proves
\eqref{scattering-profile-bounds}.
\end{proof}

A similar argument gives
\begin{equation}\label{scattering-L10/3-bound}
 \limsup_{J\to J^*}\limsup_{n\to\infty}
 \|u_n^J\|_{L_{t,x}^{10/3}}\lesssim1.
\end{equation}
Next, arguing as above, for $s\in\{0,31/60\}$ we have
\begin{align*}
 \left\|\sum_{j=1}^Jv_n^j
 \right\|_{L_t^{30/7}\dot H_x^{s,90/31}}^2
 &\lesssim
 \sum_{j=1}^J
 \|v_n^j\|_{L_t^{30/7}\dot H_x^{s,90/31}}^2\\
 &\quad+C_J\sum_{j\ne k}
 \|(|\nabla|^sv_n^j)(|\nabla|^sv_n^k)
 \|_{L_t^{15/7}L_x^{45/31}}.
\end{align*}
Using this, \eqref{scattering-auxiliary-orthogonality}, the square-sum
bound above, Strichartz, and Lemma~\ref{sobolev-equivalence}, we deduce
\begin{equation}\label{scattering-high-bound}
 \limsup_{J\to J^*}\limsup_{n\to\infty}
 \|u_n^J\|_{L_t^{30/7}H_x^{31/60,90/31}}\lesssim1.
\end{equation}

\begin{proof}[Proof of \eqref{scattering-equation-error}]
Write $F(z)=-|z|^2z$.  We first estimate
\[
 \left\||\nabla|^{1/2}
 \left[\sum_{j=1}^JF(v_n^j)
 -F\left(\sum_{j=1}^Jv_n^j\right)\right]\right\|_
 {L_t^{5/3}L_x^{30/23}}.
\]
The expression inside the brackets is a finite linear combination of
terms $v_n^jv_n^kv_n^\ell$, up to complex conjugates, where not all
three indices are equal.  Supposing $j\ne k$, the fractional product
rule gives
\begin{align*}
 &\bigl\||\nabla|^{1/2}(v_n^jv_n^kv_n^\ell)
   \bigr\|_{L_t^{5/3}L_x^{30/23}}\\
 &\quad\lesssim
 \|v_n^jv_n^k\|_{L_{t,x}^{5/2}}
 \bigl\||\nabla|^{1/2}v_n^\ell\bigr\|_{L_t^5L_x^{30/11}}\\
 &\qquad+
 \bigl\||\nabla|^{1/2}(v_n^jv_n^k)
 \bigr\|_{L_t^{5/2}L_x^{30/17}}
 \|v_n^\ell\|_{L_{t,x}^5}.
\end{align*}
The first term is $o_n(1)$ by
\eqref{scattering-profile-orthogonality}.  For the second term, the
paraproduct estimate in \cite[(3.6)]{KM-cubic} yields
\begin{align*}
 \bigl\||\nabla|^{1/2}(v_n^jv_n^k)
 \bigr\|_{L_t^{5/2}L_x^{30/17}}
 &\lesssim
 \|v_n^j|\nabla|^{1/2}v_n^k
 \|_{L_t^{5/2}L_x^{30/17}}\\
 &\quad+
 \|v_n^k|\nabla|^{1/2}v_n^j
 \|_{L_t^{5/2}L_x^{30/17}}\\
 &\quad+
 \|A(v_n^j)B(|\nabla|^{1/2}v_n^k)
 \|_{L_t^{5/2}L_x^{30/17}},
\end{align*}
where $A$ and $B$ are bounded sublinear operators that commute with
translations. The first two terms are $o_n(1)$ by
\eqref{scattering-profile-orthogonality}, while the final term is
$o_n(1)$ by the same argument used to prove
\eqref{scattering-profile-orthogonality}.

Next, we estimate the contribution of the linear remainder.  Using
\eqref{scattering-high-bound}, the fractional product rule, and Sobolev
embedding, we have
\begin{align*}
 &\limsup_{n\to\infty}
 \bigl\||\nabla|^{31/60}F(u_n^J)\bigr\|_{L_{t,x}^{10/7}}\\
 &\quad\lesssim
 \limsup_{n\to\infty}
 \|u_n^J\|_{L_t^{30/7}L_x^{45/8}}^2
 \bigl\||\nabla|^{31/60}u_n^J
 \bigr\|_{L_t^{30/7}L_x^{90/31}}
 \lesssim1
\end{align*}
uniformly in $J$.
Since $e^{-itH}r_n^J$ is uniformly bounded in
$S^1(\R)$, the same estimate gives
\begin{equation}\label{scattering-remainder-high}
 \limsup_{J\to J^*}\limsup_{n\to\infty}
 \left\||\nabla|^{31/60}
 [F(u_n^J-e^{-itH}r_n^J)-F(u_n^J)]\right\|_{L_{t,x}^{10/7}}
 \lesssim1.
\end{equation}
On the other hand, Strichartz,
\eqref{scattering-profile-bounds},
\eqref{scattering-L10/3-bound}, and
\eqref{linear-remainder-small} yield
\begin{align*}
 &\limsup_{J\to J^*}\limsup_{n\to\infty}
 \|F(u_n^J-e^{-itH}r_n^J)-F(u_n^J)\|_{L_{t,x}^{10/7}}\\
 &\quad\lesssim
 \limsup_{J\to J^*}\limsup_{n\to\infty}
 \|e^{-itH}r_n^J\|_{L_{t,x}^5}
 \bigl(\|u_n^J\|_{L_{t,x}^5}
       +\|r_n^J\|_{\dot H^{1/2}}\bigr)
 \bigl(\|u_n^J\|_{L_{t,x}^{10/3}}
       +\|r_n^J\|_{L_x^2}\bigr)=0.
\end{align*}
Interpolating this estimate with
\eqref{scattering-remainder-high}, we have
\begin{align*}
 &\left\||\nabla|^{1/2}
 [F(u_n^J-e^{-itH}r_n^J)-F(u_n^J)]\right\|_{L_{t,x}^{10/7}}\\
 &\quad\lesssim
 \|F(u_n^J-e^{-itH}r_n^J)-F(u_n^J)\|_{L_{t,x}^{10/7}}^{1/31}\\
 &\qquad\quad\times
 \left\||\nabla|^{31/60}
 [F(u_n^J-e^{-itH}r_n^J)-F(u_n^J)]
 \right\|_{L_{t,x}^{10/7}}^{30/31}.
\end{align*}
Consequently,
\[
 \limsup_{J\to J^*}\limsup_{n\to\infty}
 \left\||\nabla|^{1/2}
 [F(u_n^J-e^{-itH}r_n^J)-F(u_n^J)]\right\|_{L_{t,x}^{10/7}}
 =0.
\]
This proves
\eqref{scattering-equation-error}.
\end{proof}

Applying Lemma~\ref{stability} on $[0,\infty)$, using
\eqref{scattering-initial-approximation}--\eqref{scattering-equation-error},
we obtain
\[
 \|u(s_n+t)\|_{L_{t,x}^5([0,\infty)\times\R^3)}\lesssim1
\]
for all sufficiently large $n$.  This contradicts the failure of
forward scattering and proves that, after relabeling,
\begin{equation}\label{large-profile-inequality}
 \lim_{n\to\infty}
 M(\phi_n^1)\|\phi_n^1\|_{\dot H_V^1}^2
 \geq M(Q)\|Q\|_{\dot H^1}^2.
\end{equation}
There is no second profile satisfying
\eqref{large-profile-inequality}, since
\[
 \|Q\|_{\dot H^1}^2+\eta<2\|Q\|_{\dot H^1}^2
\]
for $\eta_c$ sufficiently small.

If $|t_n^1|\to\infty$, then Lemma~\ref{linear-L4-decay} gives
\[
 \|\phi_n^1\|_{L_x^4}\longrightarrow0.
\]
Energy decoupling, \eqref{profile-energy-coercivity}, and
\eqref{pohozaev} then imply
\[
 \limsup_{n\to\infty}\|\phi_n^1\|_{\dot H_V^1}^2
 \leq2E_0(Q)=\frac13\|Q\|_{\dot H^1}^2,
\]
contradicting \eqref{large-profile-inequality}.  Thus
$t_n^1\equiv0$.

Next, we rule out dichotomy (i.e. $J^*\geq2$). Suppose that
$J^*\geq2$. Since $\phi^2$ is nonzero, the decoupling of mass and
energy and \eqref{profile-energy-coercivity} give
\[
 \begin{cases}
 M(\phi^1)E_V(\phi^1)<M(Q)E_0(Q),&x_n^1\equiv0,\\
 M(\phi^1)E_0(\phi^1)<M(Q)E_0(Q),&|x_n^1|\to\infty.
 \end{cases}
\]
Together with \eqref{large-profile-inequality},
Lemma~\ref{variational-dichotomy} shows that \(K_V(\phi^1)<0\) in the first case and
\(K_0(\phi^1)<0\) in the second.  Let \(v^1\) be the corresponding
solution: it solves \eqref{NLSV} in the first case and \eqref{NLS0} in
the second.  Corollary~\ref{subthreshold-supercritical}, or the
blow-up/grow-up result for \eqref{NLS0} \cite{HR10}, yields a time $T$
in the lifespan of $v^1$, possibly negative, such that
\begin{equation}\label{main-profile-large}
 \begin{cases}
 \|v^1(T)\|_{\dot H_V^1}>
 \displaystyle\sup_{t\geq0}\|u(t)\|_{\dot H_V^1}+1,
 &x_n^1\equiv0,\\
 \|v^1(T)\|_{\dot H^1}>
 \displaystyle\sup_{t\geq0}\|u(t)\|_{\dot H_V^1}+1,
 &|x_n^1|\to\infty.
 \end{cases}
\end{equation}
Let $I=[0,T]$ if $T>0$ and $I=[T,0]$ if $T<0$. Since
$s_n\to\infty$, the solution $u(s_n+\cdot)$ is defined on $I$ for all
sufficiently large $n$.
To reach a contradiction under the dichotomy assumption, we will use
the linear profiles $\phi_n^j$ to build approximate solutions to
\eqref{NLSV} on $I$. Since $M(\phi_n^1)\leq M(Q)$,
\eqref{large-profile-inequality} and the Pythagorean expansions imply
that, for every fixed $J$,
\begin{equation}\label{secondary-profile-smallness}
 \limsup_{n\to\infty}
 \left\{\sum_{j=2}^J\|\phi_n^j\|_{\dot H_V^1}^2
             +\|r_n^J\|_{\dot H_V^1}^2\right\}\leq\eta.
\end{equation}

For the first profile, set $v_n^1=v^1$ if $x_n^1\equiv0$. If
$|x_n^1|\to\infty$, let $v_n^1$ be the solution given by
Lemma~\ref{nonlinear-embedding-II} on $I$.

We next construct the remaining profiles $j\geq2$ as in the proof of
\cite[Proposition~3.1]{MMZ}. First consider the case
$x_n^j\equiv0$. If $t_n^j\equiv0$, then the subthreshold scattering
theorem yields a solution $v^j$ to \eqref{NLSV} satisfying
$v^j(0)=\phi^j$ and obeying global space-time bounds. If instead
$t_n^j\to+\infty$, let $v^j$ be the solution that scatters to
$e^{-itH}\phi^j$ as $t\to-\infty$; if $t_n^j\to-\infty$, then let $v^j$ be the solution that scatters to
$e^{-itH}\phi^j$ as $t\to+\infty$. In either case, we define
\[
 v_n^j(t,x)=v^j(t-t_n^j,x).
\]
To verify the hypotheses, use mass decoupling,
$E_V(f)\leq\frac12\|f\|_{\dot H_V^1}^2$, and
\eqref{secondary-profile-smallness}. After decreasing $\eta_c$,
\[
 \begin{aligned}
 \limsup_{n\to\infty}M(\phi_n^j)E_V(\phi_n^j)
 &\leq\frac12M(Q)\eta<M(Q)E_0(Q),\\
 \limsup_{n\to\infty}M(\phi_n^j)\|\phi_n^j\|_{\dot H_V^1}^2
 &\leq M(Q)\eta<M(Q)\|Q\|_{\dot H^1}^2.
 \end{aligned}
\]
When $|t_n^j|\to\infty$, Lemma~\ref{linear-L4-decay} gives
\[
 E_V(\phi_n^j)\longrightarrow
 \frac12\|\phi^j\|_{\dot H_V^1}^2.
\]
The scattering condition gives
$M(v^j)=M(\phi^j)$,
$E_V(v^j)=\frac12\|\phi^j\|_{\dot H_V^1}^2$, and
$\|v^j(-t_n^j)-e^{it_n^jH}\phi^j\|_{H^1}\to0$.
Thus $v^j(-t_n^j)$ satisfies both subthreshold inequalities for large
$n$, and Theorem~\ref{subthreshold-scattering} applies at that time.

Next, if $|x_n^j|\to\infty$, we appeal to
Lemma~\ref{nonlinear-embedding-I} to obtain a global solution $v_n^j$
to \eqref{NLSV}. The hypotheses of that lemma
follow from the preceding two inequalities when $t_n^j\equiv0$; when
$|t_n^j|\to\infty$, mass decoupling and
\eqref{secondary-profile-smallness} give, after decreasing $\eta_c$,
\[
 \frac12\|\phi^j\|_{L_x^2}^2\|\phi^j\|_{\dot H^1}^2
 \leq\frac12M(Q)\eta<M(Q)E_0(Q).
\]

By construction, for every fixed $j$,
\begin{equation}\label{nonlinear-profile-initial-match}
 \|v_n^j(0)-\phi_n^j\|_{H^1}\longrightarrow0.
\end{equation}
All profiles with $j\geq2$ are global and uniformly bounded in
$S^1(\R)$, while the first profile is uniformly bounded in $S^1(I)$.

We now define approximate solutions to \eqref{NLSV} by
\[
 u_n^J(t,x)=\sum_{j=1}^Jv_n^j(t,x)+e^{-itH}r_n^J
\]
and immediately observe that by construction,
\begin{equation}\label{nonlinear-initial-approximation}
 \lim_{n\to\infty}\|u_n^J(0)-u_n\|_{H^1}=0
 \qquad\text{for each }J.
\end{equation}
Our next goal is to show
\begin{align}
 &\limsup_{J\to J^*}\limsup_{n\to\infty}
 \left\{\|u_n^J(0)\|_{H^1}
 +\|u_n^J\|_{L_{t,x}^5(I\times\R^3)}\right\}
 \lesssim1,\label{nonlinear-profile-bounds}\\
 &\limsup_{J\to J^*}\limsup_{n\to\infty}
 \bigl\|(i\partial_t-H)u_n^J+|u_n^J|^2u_n^J
 \bigr\|_{\dot N^{1/2}(I)}=0.
 \label{nonlinear-equation-error}
\end{align}

The proofs of \eqref{scattering-profile-bounds} and
\eqref{scattering-equation-error} apply with $\R$ replaced by $I$.
We first verify the corresponding bounds on $I$. For $j\geq2$, the
nonlinear profiles are constructed as before. If $|x_n^1|\to\infty$,
Lemma~\ref{nonlinear-embedding-II} gives a uniform $S^1(I)$ bound
and approximation by smooth compactly supported functions in all
the spaces used in
\eqref{scattering-profile-orthogonality}--\eqref{scattering-auxiliary-orthogonality}.
If $x_n^1\equiv0$, these properties follow from the local theory
for $v^1$. Parameter orthogonality therefore gives the same
orthogonality estimates on $I$.

The small-data argument for sufficiently large $j$ is unchanged.
The remaining profiles are finite in number and have uniformly
bounded $S^1(I)$ norms. Consequently,
\[
 \limsup_{J\to J^*}\limsup_{n\to\infty}
 \sum_{j=1}^J\|v_n^j\|_{S^1(I)}^2\lesssim1.
\]
It follows, exactly as in
\eqref{scattering-profile-bounds}--\eqref{scattering-high-bound}, that
\[
 \begin{aligned}
 &\limsup_{J\to J^*}\limsup_{n\to\infty}
 \Bigl\{
 \|u_n^J\|_{L_{t,x}^5(I\times\R^3)}
 +\|u_n^J\|_{L_{t,x}^{10/3}(I\times\R^3)}\\
 &\qquad\qquad
 +\|u_n^J\|_{L_t^{30/7}H_x^{31/60,90/31}(I\times\R^3)}
 \Bigr\}\lesssim1.
 \end{aligned}
\]
Together with \eqref{nonlinear-initial-approximation}, this proves
\eqref{nonlinear-profile-bounds}. The estimates for the cubic
interactions and the linear remainder, including interpolation
between orders $0$ and $31/60$, now give
\eqref{nonlinear-equation-error} by the proof of
\eqref{scattering-equation-error}. In both arguments $J$ is fixed
before $n\to\infty$, and then $J\to J^*$.

Lemma~\ref{stability}, together with
\eqref{nonlinear-initial-approximation},
\eqref{nonlinear-profile-bounds}, and
\eqref{nonlinear-equation-error}, gives
\begin{equation}\label{nonlinear-stability}
 \lim_{J\to J^*}\limsup_{n\to\infty}
 \|u(s_n+\cdot)-u_n^J\|_{\dot S^{1/2}(I)}=0.
\end{equation}
We now pass to the frame of the first profile.  If
$|x_n^1|\to\infty$, \eqref{local-embedding-pointwise} gives
\[
 v_n^1(T,\cdot+x_n^1)\longrightarrow v^1(T)
 \qquad\text{in }\dot H^{1/2},
\]
while this sequence is bounded in $H^1$. Thus the convergence is weak
in $H^1$ and, by Sobolev embedding and interpolation, strong in $L^4$.
Conservation of mass also gives strong convergence in $L^2$. Moreover,
\[
 E_V(v_n^1(T))
 =E_V(\phi^1(\cdot-x_n^1))
 \longrightarrow E_0(\phi^1)=E_0(v^1(T)).
\]
Consequently,
\[
 \|\nabla v_n^1(T)\|_{L_x^2}^2
 +\int_{\R^3}V(x)|v_n^1(T,x)|^2\,dx
 \longrightarrow \|\nabla v^1(T)\|_{L_x^2}^2.
\]
Since $V\geq0$, weak lower semicontinuity yields
\[
 \int_{\R^3}V(x)|v_n^1(T,x)|^2\,dx\longrightarrow0
\]
and
\[
 v_n^1(T,\cdot+x_n^1)\longrightarrow v^1(T)
 \qquad\text{strongly in }H^1.
\]
The same conclusion is immediate when $x_n^1\equiv0$. For $j\geq2$,
parameter orthogonality, scattering,
and the construction and stability argument in the proof of
Lemma~\ref{nonlinear-embedding-I} give
\[
 v_n^j(T,\cdot+x_n^1)\rightharpoonup0
 \qquad\text{weakly in }H^1.
\]
Writing $r_n^J=r_n^1-\sum_{j=2}^J\phi_n^j$, the weak convergence in
\eqref{linear-remainder-weak} with $J=1$, parameter orthogonality, and
the convergence of the translated propagators show that
$e^{-iTH}r_n^J$ converges weakly to zero in the frame of the first
profile; see \cite[Lemma~2.11]{KMVZ}. 

We now identify the weak
limit of the exact solutions.  After passing to a subsequence, the sequence
$u(s_n+T)$, translated by $x_n^1$ when $|x_n^1|\to\infty$, has a weak
$H^1$ limit.  For each fixed $J$, the preceding convergences give
$u_n^J(T,\cdot+x_n^1)\rightharpoonup v^1(T)$ in $H^1$.
Since $V\geq0$, one has
$\||\nabla|^{1/2}f\|_2\leq\|H^{1/4}f\|_2$.
Thus \eqref{nonlinear-stability} controls the error at $T$ in
the $\dot H^{1/2}$ norm. Testing against a smooth compactly supported function,
which belongs to $\dot H^{-1/2}$, and passing first to $n\to\infty$
and then to $J\to J^*$ identifies the distributional limit with
$v^1(T)$. The uniform $H^1$ bound then gives the same weak $H^1$ limit.

Consequently,
\[
 \begin{cases}
 u(s_n+T)\rightharpoonup v^1(T),&x_n^1\equiv0,\\
 u(s_n+T,\cdot+x_n^1)\rightharpoonup v^1(T),&|x_n^1|\to\infty,
 \end{cases}
 \qquad\text{weakly in }H^1.
\]
If $x_n^1\equiv0$, weak lower semicontinuity gives
\[
 \|v^1(T)\|_{\dot H_V^1}
 \leq\sup_{t\geq0}\|u(t)\|_{\dot H_V^1}.
\]
If $|x_n^1|\to\infty$, it gives
\[
 \|v^1(T)\|_{\dot H^1}
 \leq\sup_{t\geq0}\|u(t)\|_{\dot H_V^1}.
\]
This contradicts \eqref{main-profile-large} and hence rules out the
possibility of dichotomy.

Having ruled out vanishing and dichotomy, we are left with the
conclusion that $J^*=1$ (`compactness'). In particular, since
$t_n^1\equiv0$, our decomposition reduces to
\[
 u_n=\phi_n^1+r_n^1=\phi^1(\cdot-x_n^1)+r_n^1.
\]
We claim that $r_n^1\to0$ strongly in $H^1$. Indeed, if strong
convergence fails, then mass and energy decoupling and
\eqref{profile-energy-coercivity} show that the first profile lies
strictly below the mass--energy threshold. Let $v^1$ denote the
solution to \eqref{NLSV} with initial data $\phi^1$ when
$x_n^1\equiv0$, and the solution to \eqref{NLS0} with initial data
$\phi^1$ when $|x_n^1|\to\infty$. As above, there is a time $T$ in
the lifespan of $v^1$, possibly negative, such that
\[
 \begin{cases}
 \|v^1(T)\|_{\dot H_V^1}>
 \displaystyle\sup_{t\geq0}\|u(t)\|_{\dot H_V^1}+1,
 &x_n^1\equiv0,\\
 \|v^1(T)\|_{\dot H^1}>
 \displaystyle\sup_{t\geq0}\|u(t)\|_{\dot H_V^1}+1,
 &|x_n^1|\to\infty.
 \end{cases}
\]
Let $I=[0,T]$ if $T>0$ and $I=[T,0]$ if $T<0$.  Set
$v_n^1=v^1$ when $x_n^1\equiv0$, and when $|x_n^1|\to\infty$ let
$v_n^1$ be the solution furnished by
Lemma~\ref{nonlinear-embedding-II} on $I$. Since $s_n\to\infty$,
$u(s_n+\cdot)$ is defined on $I$ for all sufficiently large $n$. The
preceding stability argument on $I$, now with only $v_n^1$ and
$e^{-itH}r_n^1$, gives
\[
 \begin{cases}
 u(s_n+T)\rightharpoonup v^1(T),&x_n^1\equiv0,\\
 u(s_n+T,\cdot+x_n^1)\rightharpoonup v^1(T),
 &|x_n^1|\to\infty,
 \end{cases}
 \qquad\text{weakly in }H^1.
\]
Since $s_n+T\geq0$ for all sufficiently large $n$, weak lower
semicontinuity contradicts the preceding choice of $T$. Thus
\begin{equation}\label{one-profile-convergence}
 \|u_n-\phi^1(\cdot-x_n^1)\|_{H^1}\longrightarrow0.
\end{equation}
Equivalently, no further profile occurs and $r_n^1\to0$ strongly in
$H^1$; hence the decomposition contains only one profile.

It remains to rule out $|x_n^1|\to\infty$. Suppose that
\(|x_n^1|\to\infty\). Then \eqref{fixed-delta-identities} and
\eqref{one-profile-convergence} imply
\[
 M(\phi^1)=M(Q),\qquad E_0(\phi^1)=E_0(Q),\qquad
 \|\phi^1\|_{\dot H^1}^2=\|Q\|_{\dot H^1}^2+\eta.
\]
Let \(v^1\) be the solution to \eqref{NLS0} with \(v^1(0)=\phi^1\).
By Theorem~\ref{free-blowup-growup}, there exists a time $T$ in the
lifespan of $v^1$, which may be
positive or negative, such that
\[
 \|v^1(T)\|_{\dot H^1}>
 \sup_{t\geq0}\|u(t)\|_{\dot H_V^1}+1.
\]
If $T>0$, apply Lemma~\ref{nonlinear-embedding-II} on $[0,T]$; if
$T<0$, apply it on $[T,0]$.  Denote the resulting solution by $v_n^1$.
Since $s_n\to\infty$, one has $s_n+T\geq0$ for all sufficiently large
$n$.
By \eqref{local-embedding-pointwise}, stability, and
\eqref{one-profile-convergence},
\begin{align}\label{Chc}
	 u(s_n+T,\cdot+x_n^1)\longrightarrow v^1(T)
 \qquad\text{in }\dot H^{1/2}(\R^3).
\end{align}
In view of \eqref{Chc} and the fact that the sequence is bounded in $H^1$, it converges weakly to $v^1(T)$
in $H^1$. By Sobolev embedding, the convergence in $\dot H^{1/2}$
implies strong convergence in $L^3$, while the $H^1$ bound gives a
uniform $L^6$ bound. Interpolation therefore yields
\[
u(s_n+T,\cdot+x_n^1)\longrightarrow v^1(T)
\qquad\text{strongly in }L^4.
\]
Moreover,
$M(u(s_n+T))=M(Q)=M(v^1(T))$, so weak $L^2$ convergence and equality
of the norms give strong $L^2$ convergence.
Conservation of energy gives
\[
 \|\nabla u(s_n+T)\|_{L_x^2}^2
 +\int_{\R^3}V(x)|u(s_n+T,x)|^2\,dx
 \longrightarrow\|\nabla v^1(T)\|_{L_x^2}^2.
\]
Since $V\geq0$, weak lower semicontinuity yields
\[
 \int_{\R^3}V(x)|u(s_n+T,x)|^2\,dx\longrightarrow0
\]
and
\[
 u(s_n+T,\cdot+x_n^1)\longrightarrow v^1(T)
 \qquad\text{strongly in }H^1(\R^3),
\]
and hence
$\|v^1(T)\|_{\dot H^1}
\leq\sup_{t\geq0}\|u(t)\|_{\dot H_V^1}$, which is a contradiction.
Hence
\(x_n^1\equiv0\), and \eqref{one-profile-convergence} proves the
proposition for $\mu=2$.

For $1<\mu<2$, we use Lemma~\ref{stability-inverse-power} in place of Lemma~\ref{stability}. Let $\mathcal I$ denote the interval where stability is applied; it remains to upgrade convergence in $\dot N^{1/2}(\mathcal I)$ to $N^{1/2}(\mathcal I)$. Write
\[
\mathcal P_n^J:=\sum_{j=1}^JF(v_n^j)-F\!\left(\sum_{j=1}^Jv_n^j\right),\qquad
\mathcal R_n^J:=F(u_n^J-e^{-itH}r_n^J)-F(u_n^J).
\]
Every mixed term in $\mathcal P_n^J$ has the form $v_n^jv_n^kv_n^\ell$, up to conjugates, with at least two distinct indices. If $j\ne k$,
\[
\|v_n^jv_n^kv_n^\ell\|_{L_t^{5/3}L_x^{30/23}(\mathcal I)}
\lesssim\|v_n^jv_n^k\|_{L_{t,x}^{5/2}(\mathcal I)}\|v_n^\ell\|_{L_t^5L_x^{30/11}(\mathcal I)}\to0.
\]
Thus parameter orthogonality and the argument giving convergence in $\dot N^{1/2}(\mathcal I)$ yield
\[
\limsup_{J\to J^*}\limsup_{n\to\infty}\|\mathcal P_n^J\|_{L_t^{5/3}L_x^{30/23}(\mathcal I)}=0.
\]
Similarly, Hölder's inequality, the uniform bounds for $u_n^J$, and the smallness of $e^{-itH}r_n^J$ give
\[
\limsup_{J\to J^*}\limsup_{n\to\infty}\|\mathcal R_n^J\|_{L_{t,x}^{10/7}(\mathcal I)}=0.
\]
Together with convergence in $\dot N^{1/2}(\mathcal I)$, these imply
\[
\limsup_{J\to J^*}\limsup_{n\to\infty}\|(i\partial_t-H)u_n^J+|u_n^J|^2u_n^J\|_{N^{1/2}(\mathcal I)}=0.
\]
The $H^{1/2}$ convergence at the initial time follows from the established $H^1$ convergence. Hence Lemma~\ref{stability-inverse-power} applies in each preceding stability argument, and the rest of the proof is unchanged. 
\end{proof}

\section{A rigidity result}\label{S:rigidity}

\begin{theorem}\label{rigidity}
There exists $\eta_*>0$ such that there is no forward-global solution to
\eqref{NLSV} satisfying \eqref{normalized-threshold} and
\begin{equation}\label{rigidity-hypothesis}
 0<\delta(u(t))\leq\eta_*,\qquad t\geq0.
\end{equation}
\end{theorem}

In the remainder of this section, suppose that $u$ satisfies the
assumptions of Theorem~\ref{rigidity}, with $\eta_*<\eta_0$. By
Proposition~\ref{modulation}, there exist $C^1$ functions
\[
 \theta:[0,\infty)\to\R,\qquad
 y:[0,\infty)\to\R^3,\qquad
 \alpha:[0,\infty)\to\R
\]
such that \eqref{decompz0} and
\eqref{mo10}--\eqref{center-lower-bound} hold. In particular,
\[
 \|u(t)-e^{i(t+\theta(t))}Q(\cdot-y(t))\|_{H^1}
 \lesssim\delta(t),\qquad t\geq0.
\]

\begin{proposition}\label{bootstrap}
If $|y(0)|$ is sufficiently large and $\eta_*$ is sufficiently small,
then
\[
 |y(t)-y(0)|\leq1,\qquad t\geq0.
\]
\end{proposition}

\begin{proof}
We prove this by a bootstrap argument. It is enough to show that on
every interval $[0,T]$ satisfying
\[
 |y(t)-y(0)|\leq2,\qquad 0\leq t\leq T,
\]
one has
\[
 |y(t)-y(0)|\leq1,\qquad 0\leq t\leq T.
\]

Let $w_R(x)=R^2\phi(|x|/R)$, with $\phi$ as in
\eqref{phi-virial}, and set
\[
 I(t)=\int_{\R^3}w_R(x-y(0))|u(t,x)|^2\,dx.
\]
For $f\in H^1(\R^3)$, write
\begin{align*}
 \mathcal G(f)
 &=4\Re\sum_{j,k=1}^3\int_{\R^3}
   \bigl((w_R)_{jk}-2\delta_{jk}\bigr)
   \overline{\partial_jf}\,\partial_kf\,dx\\
 &\quad+\int_{\R^3}(6-\Delta w_R)|f|^4\,dx
       -\int_{\R^3}\Delta^2w_R|f|^2\,dx.
\end{align*}
The truncated virial identity for the standing wave
$e^{it+i\gamma}Q(x-z)$ gives
\[
 \mathcal G(e^{i\gamma}Q(\cdot-z))=0.
\]
Applying \eqref{virial} and using
\[
 8\|u(t)\|_{\dot H^1}^2-6\|u(t)\|_{L^4}^4
 =-4\delta(t)-8\int_{\R^3}V|u(t)|^2\,dx,
\]
we obtain
\begin{align*}
 I''(t)
 &=-4\delta(t)
   +2a\int_{\R^3}
   \left[
   \mu\frac{x\cdot\nabla w_R(x-y(0))}{|x|^{\mu+2}}
   -\frac4{|x|^\mu}
   \right]|u(t,x)|^2\,dx\notag\\
 &\quad
   +\mathcal G(u(t,\cdot+y(0)))\notag\\
 &\quad
   -\mathcal G\bigl(e^{i(t+\theta(t))}
        Q(\cdot+y(0)-y(t))\bigr).
\end{align*}

We next choose $R$ sufficiently large, independently of $T$, so that
\begin{align}
 \left|\mathcal G(u(t,\cdot+y(0)))
 -\mathcal G\bigl(e^{i(t+\theta(t))}
 Q(\cdot+y(0)-y(t))\bigr)\right|
       &\leq\delta(t),\label{bst1}\\
 2a\left|\int_{\R^3}
 \left[
 \mu\frac{x\cdot\nabla w_R(x-y(0))}{|x|^{\mu+2}}
 -\frac4{|x|^\mu}
 \right]|u(t,x)|^2\,dx\right|
 &\leq\delta(t),\label{bst2}
\end{align}
for all $t\in[0,T]$.

Taking these estimates for granted, the preceding identity gives
\[
 I''(t)\leq-2\delta(t).
\]
Moreover, \eqref{virial-first} and \eqref{mo10} yield
\[
 |I'(t)|\lesssim R\delta(t).
\]
Consequently,
\[
 \int_0^T\delta(t)\,dt
 \lesssim R[\delta(0)+\delta(T)]
 \lesssim R\eta_*.
\]
It follows from \eqref{mo20} that
\[
 |y(t)-y(0)|
 \leq\int_0^t|y'(s)|\,ds
 \lesssim\int_0^T\delta(s)\,ds
 \lesssim R\eta_*,\qquad 0\leq t\leq T.
\]
Taking $\eta_*$ sufficiently small gives the desired bound.

It remains to prove \eqref{bst1} and \eqref{bst2}. Given
$\varepsilon>0$, choose $R$ sufficiently large that the $H^1$ and
$L^4$ norms of $Q$ on $\{|x|>R-2\}$ are at most $\varepsilon$. Since
$|y(t)-y(0)|\leq2$ on $[0,T]$,
\[
 \{|x-y(0)|>R\}\subset\{|x-y(t)|>R-2\}.
\]
Hence \eqref{mo10} gives
\[
 \|u(t)\|_{H^1(|x-y(0)|>R)}
 \lesssim\varepsilon+\delta(t),\qquad
 \|Q(\cdot-y(t))\|_{H^1(|x-y(0)|>R)}
 \lesssim\varepsilon,
\]
and the same bounds hold in $L^4$. The difference of the gradient
terms in \eqref{bst1} is therefore bounded by
$C(\varepsilon+\delta(t))\delta(t)$, while
\begin{align*}
 &\int_{|x-y(0)|>R}
 \bigl||u(t,x)|^4-|Q(x-y(t))|^4\bigr|\,dx\\
 &\quad\lesssim
 \left[\|u(t)\|_{L^4(|x-y(0)|>R)}^3
 +\|Q(\cdot-y(t))\|_{L^4(|x-y(0)|>R)}^3\right]\\
 &\qquad\times
 \|u(t)-e^{i(t+\theta(t))}Q(\cdot-y(t))\|_{L^4}\\
 &\quad\lesssim(\varepsilon+\delta(t))^3\delta(t).
\end{align*}
The term containing $\Delta^2w_R$ is estimated similarly. Choosing
$\varepsilon$ and then $\eta_*$ sufficiently small proves
\eqref{bst1}.

Finally, suppose that $|y(0)|\geq6R$. On the support of
$\nabla w_R(x-y(0))$ one has $|x-y(0)|\leq3R$, and hence
$|x|\geq3R$. Since $|\nabla w_R|\lesssim R$, the left-hand side of
\eqref{bst2} is bounded by
\[
 C\int_{\R^3}V(x)|u(t,x)|^2\,dx
 \lesssim\delta(t)^2
\]
by \eqref{potential-modulation}. Therefore, 
\eqref{bst2} holds for sufficiently small $\eta_{\ast}\ll1$, and the proposition follows.
\end{proof}

\begin{proof}[Proof of Theorem~\ref{rigidity}]
By \eqref{center-lower-bound}, after decreasing $\eta_*$ we have
$|y(0)|$ sufficiently large. Proposition~\ref{bootstrap} then gives
\[
 |y(t)|\leq|y(0)|+1,
 \qquad t\geq0.
\]
On the other hand, Corollary~\ref{sequencec} gives a sequence
$t_n\to\infty$ such that $\delta(u(t_n))\to0$. Applying
\eqref{center-lower-bound} once more, we obtain $|y(t_n)|\to\infty$,
which is a contradiction.
\end{proof}

\section{Proof of the main result}\label{S:reduction}

\begin{proof}[Proof of Proposition~\ref{prop1.7}]
Suppose, to the contrary, that $u$ is a forward-global solution
satisfying \eqref{normalized-threshold} and \eqref{bounded-solution}.
By Corollary~\ref{sequencec}, there exists $t_n\to\infty$ such that
\[
 \delta(u(t_n))\longrightarrow0.
\]
Let $\eta_*>0$ be as in Theorem~\ref{rigidity}, decreased if necessary
so that Proposition~\ref{P:compact} applies with $\eta=\eta_*$.

We split into two cases:
\begin{itemize}
\item[(i)]
There exists $t_0>0$ such that
\[
 \sup_{t\geq t_0}\delta(u(t))\leq\eta_*.
\]
\item[(ii)]
There exists a sequence $t_n^-\to\infty$ such that
\[
 \delta(u(t_n^-))>\eta_*
 \qquad\text{for all }n.
\]
\end{itemize}

Case (i), after a time translation, contradicts
Theorem~\ref{rigidity}. We therefore assume that case (ii) occurs.
Passing to a subsequence, we may assume that
\[
 t_n^-<t_n<t_{n+1}^-,
 \qquad \delta(u(t_n))<\eta_*.
\]
Set
\[
 t_n^+=\max\bigl\{t\in[t_n^-,t_n]:
                   \delta(u(t))=\eta_*\bigr\}.
\]
By continuity and Lemma~\ref{threshold-sign},
\[
 t_n^+\to\infty,
 \qquad \delta(u(t_n^+))=\eta_*,
 \qquad 0<\delta(u(t))<\eta_*
 \quad(t_n^+<t\leq t_n).
\]

Proposition~\ref{P:compact}, applied to $\{u(t_n^+)\}$, gives, after
passing to a subsequence, $w_0\in H^1(\R^3)$ such that
\[
 u(t_n^+)\longrightarrow w_0
 \qquad\text{in }H^1(\R^3).
\]
We now let $w$ be the maximal-lifespan solution to \eqref{NLSV}
with $w(0)=w_0$.  The preceding convergence and conservation of
mass and energy give
\[
 M(w)=M(Q),\qquad E_V(w)=E_0(Q),
 \qquad \delta(w(0))=\eta_*.
\]
In particular, Lemma~\ref{threshold-sign} shows that $w$ satisfies
\eqref{normalized-threshold} and that $\delta(w(t))>0$ throughout
its lifespan.

To see that $w$ is forward global, we define
\[
 \widetilde u_n(t,x)=u(t_n^++t,x).
\]
Fix $0<T<T_+(w)$. By the local theory, $w$ has finite
$L_{t,x}^5$ norm on $[0,T]\times\R^3$. Then we can apply Lemma~\ref{stability} when $\mu=2$, or
Lemma~\ref{stability-inverse-power} when $1<\mu<2$ to obtain that
\begin{equation}\label{cov1}
     \sup_{0\leq t\leq T}
 \|\widetilde u_n(t)-w(t)\|_{\dot H^{1/2}}
 \longrightarrow0.
\end{equation}
Together with \eqref{bounded-solution}, this implies that, for
every $t\in[0,T]$,
\begin{equation}\label{cov2}
    \widetilde u_n(t)\rightharpoonup w(t)
 \qquad\text{weakly in }H^1(\R^3).
\end{equation}
By weak lower semicontinuity,
\[
 \|w(t)\|_{H^1}
 \leq\liminf_{n\to\infty}\|\widetilde u_n(t)\|_{H^1}
 \leq\sup_{s\geq0}\|u(s)\|_{H^1}.
\]
Since $T<T_+(w)$ was arbitrary, the continuation criterion in
Proposition~\ref{local-theory} shows that $w$ is forward global
and bounded in $H^1$. Therefore, \eqref{cov1}  and \eqref{cov2} hold
on every fixed interval $[0,T]\subset[0,\infty)$.

We next claim that
\[
 \sup_{t\geq0}\delta(w(t))\leq\eta_*.
\]
Together with Theorem~\ref{rigidity}, this gives the desired
contradiction. To this end, we first show that, for any fixed
$T>0$, there exists $n_0(T)$ such that
\[
 t_n-t_n^+\geq T
\]
for all $n\geq n_0(T)$. Suppose this fails. After passing to a
subsequence, we may assume that
\[
 \tau_n:=t_n-t_n^+<T,\qquad \tau_n\longrightarrow\tau\in[0,T].
\]
By \eqref{cov2}, we conclude that
\[
 u(t_n)=\widetilde u_n(\tau_n)\rightharpoonup w(\tau)
 \qquad\text{weakly in }H^1(\R^3).
\]
Since the quadratic form $f\mapsto\|f\|_{\dot H_V^1}^2$ is
weakly lower semicontinuous on $H^1$, we obtain
\[
 0<\delta(w(\tau))
 \leq\liminf_{n\to\infty}\delta(u(t_n))=0,
\]
which is a contradiction (as $w\neq0$). Therefore,
\[
 t_n-t_n^+\longrightarrow+\infty.
\]

Now fix $T>0$. For all sufficiently large $n$, we have
$t_n^++T\in(t_n^+,t_n]$, and hence
\[
 \delta(\widetilde u_n(T))<\eta_*.
\]
On the other hand, again by \eqref{cov2},
\[
 \widetilde u_n(T)\rightharpoonup w(T)
 \qquad\text{weakly in }H^1(\R^3).
\]
By weak lower semicontinuity,
\[
 0<\delta(w(T))
 \leq\liminf_{n\to\infty}\delta(\widetilde u_n(T))
 \leq\eta_*.
\]
Since $T>0$ was arbitrary and $\delta(w(0))=\eta_*$, the claim
follows. This contradicts Theorem~\ref{rigidity} and completes
the proof.
\end{proof}

\begin{proof}[Proof of Theorem~\ref{main}]
Let $u$ be a maximal-lifespan solution satisfying
\eqref{normalized-threshold}. In the positive time direction, either
the maximal lifespan is finite or $u$ is forward global. In the latter
case, Proposition~\ref{prop1.7} shows that
\[
 \limsup_{t\to\infty}\|u(t)\|_{\dot H^1}=\infty.
\]
Thus $u$ either blows up or grows up in positive time. Applying the
same argument to the time-reversed solution $\overline{u(-t)}$ gives
the corresponding alternative in negative time. The four possible
combinations are precisely those stated in Theorem~\ref{main}.
\end{proof}

\appendix
\section{Finite-variance blow-up}\label{finite-variance-appendix}

In this appendix, we extend the finite-variance blow-up result of [1, Theorem 1.3(ii)] to the endpoint \(\mu=2\) and give a unified proof for \(1<\mu\le2\). More precisely, we prove the following theorem.
\begin{theorem}[Finite-variance blow-up]\label{finite-variance}
Let $V(x)=a|x|^{-\mu}$ with $a>0$ and $1<\mu\leq2$. Suppose that
$u_0\in H^1(\mathbb R^3)$ satisfies
\[
M(u_0)E_V(u_0)=M(Q)E_0(Q),\qquad K_V(u_0)<0.
\]
If $|x|u_0\in L^2(\mathbb R^3)$, then the solution to \eqref{NLSV} blows
up in both time directions.
\end{theorem}

\begin{proof}[Proof of Theorem~\ref{finite-variance}]

The scaling in
Corollary~\ref{product-threshold} preserves finite variance, so it is
enough to prove this theorem under the assumption \eqref{normalized-threshold}.

Suppose, towards a contradiction, that $u$ is forward global.  Finite
variance persists on compact time intervals.  Thus
\[
 I(t)=\int_{\R^3}|x|^2|u(t,x)|^2\,dx
\]
is finite for every $t\geq0$, and the full virial identity gives
\begin{equation}\label{full-negative-virial}
 I''(t)=8K_V(u(t))
       =-4\delta(t)-4(2-\mu)\int_{\R^3}V|u(t)|^2\,dx
       \leq-4\delta(t)<0.
\end{equation}
Since $I(t)\geq0$, one must have $I'(t)\geq0$ for every $t\geq0$.
Indeed, if $I'(t_0)<0$, then \eqref{full-negative-virial} gives
$I'(t)\leq I'(t_0)<0$ for $t\geq t_0$, which makes $I(t)$ negative for
large $t$.  Integrating \eqref{full-negative-virial}, we obtain
\begin{equation}\label{finite-variance-delta-integrable}
 4\int_0^T\delta(t)\,dt
 \leq I'(0)-I'(T)\leq I'(0),\qquad T>0.
\end{equation}
It follows that $\delta\in L^1(0,\infty)$, and hence there is a sequence
$t_n\to\infty$ such that $\delta(t_n)\to0$.

Let $\eta_*$ be as in Theorem~\ref{rigidity}.  Proposition
\ref{modulation} gives constants $c,C>0$ such that
\begin{equation}\label{appendix-modulation-bounds}
 \alpha(t)\geq c\delta(t),\qquad
 |\alpha'(t)|\leq C\delta(t)
\end{equation}
whenever $0<\delta(t)\leq\eta_*$.  We claim that, for all sufficiently
large $n$,
\begin{equation}\label{appendix-eventual-modulation}
 \delta(t)\leq\eta_*,\qquad t\geq t_n.
\end{equation}
Otherwise, after taking $n$ large enough that $\delta(t_n)<\eta_*$, let
$s_n>t_n$ be the first time for which $\delta(s_n)=\eta_*$.  The
modulation parameters are defined on $[t_n,s_n]$, and
$\alpha(t_n)\to0$.  Therefore, for large $n$,
\[
 \frac{c\eta_*}{2}
 \leq \alpha(s_n)-\alpha(t_n)
 \leq\int_{t_n}^{s_n}|\alpha'(t)|\,dt
 \leq C\int_{t_n}^{\infty}\delta(t)\,dt,
\]
which tends to zero by \eqref{finite-variance-delta-integrable}.  This
contradiction proves \eqref{appendix-eventual-modulation}.

After translating time by one sufficiently large $t_n$, the solution
satisfies \eqref{rigidity-hypothesis} for all positive times. This
contradicts Theorem~\ref{rigidity}; hence the positive lifespan endpoint
is finite.  Applying the same argument to $\overline{u(-t)}$ proves
finite-time blow-up in the negative direction as well.
\end{proof}

\end{document}